\documentclass[a4paper, 11pt]{amsart}
\usepackage{graphicx, euscript, enumitem,kantlipsum}
\usepackage{amssymb}
\usepackage{amsmath}
\usepackage{amsthm}
\usepackage{amscd}
\usepackage{mathbbol}
\usepackage[all,2cell]{xy}

\usepackage[pagebackref,colorlinks]{hyperref}

\usepackage{tikz-cd}
\usetikzlibrary{arrows,backgrounds,shadows,decorations.markings,matrix}

\usepackage{geometry}
\UseAllTwocells \SilentMatrices
\newtheorem{thm}{Theorem}[section]

\newtheorem{lem}[thm]{Lemma}
\newtheorem{exm}[thm]{Example}

\newtheorem{thmy}{Theorem}

\newtheorem{prop}[thm]{Proposition}
\theoremstyle{definition}
\newtheorem{defn}[thm]{Definition}
\theoremstyle{remark}
\newtheorem{rem}[thm]{\bf Remark}
\numberwithin{equation}{section}

\makeatletter
\newcommand{\smallotimes}{\mathbin{\mathpalette\make@small\otimes}}

\newcommand{\make@small}[2]{%
  \vcenter{\hbox{%
    $\m@th\ifx#1\displaystyle\scriptstyle\else\ifx#1\textstyle\scriptstyle
     \else\scriptscriptstyle\fi\fi#2$%
  }}%
}
\makeatother

\begin{document}

\title[Higher-dimensional module factorizations and complete intersections]{Higher-dimensional module factorizations and complete intersections}

\author{Xiao-Wu Chen}

\makeatletter
\@namedef{subjclassname@2020}{\textup{2020} Mathematics Subject Classification}
\makeatother


\date{\today}

\begin{abstract}
We introduce higher-dimensional module factorizations associated to a regular sequence. They include higher-dimensional matrix factorizations, which are commutative cubes consisting of free modules with edges being classical matrix factorizations.  We characterize the stable category of maximal Cohen-Macaulay modules over a complete intersection via higher-dimensional matrix factorizations over the corresponding regular local ring. The result generalizes to  noncommutative rings, including quantum complete intersections.
\end{abstract}

\maketitle


\section{Introduction}

 \subsection{Higher-dimensional matrix factorizations} Let $S$ be a regular local ring with a regular sequence $(\omega_1, \cdots, \omega_n)$. The corresponding quotient ring $R=S/{(\omega_1, \cdots, \omega_n)}$ is called a complete intersection. If $n=1$, the ring $R$ is called a hypersurface.

 Higher syzygies of $R$-modules play a central role in commutative algebra \cite{Sha, Eis}. They boil down to the rich structure on  the category  ${\rm MCM}(R)$ of maximal Cohen-Macaulay (MCM for short) $R$-modules.  To be more precise, ${\rm MCM}(R)$ is a Frobenius exact category in the sense of \cite{Hel60, Hap}, whose stable category $\underline{\rm MCM}(R)$ is canonically triangulated \cite{Hap}. The category  $\underline{\rm MCM}(R)$ detects the singularity of $R$, since it is triangle equivalent to the singularity category \cite{Buc, Orl} of $R$.

 In the hypersurface case, that is, $R=S/{(\omega)}$ for a nonzero element $\omega$, the category $\underline{\rm MCM}(R)$ is triangle equivalent to the stable category of $\mathbf{MF}(S; \omega)$, the Frobenius exact category of matrix factorizations \cite{Eis} of $\omega$. Recall that a \emph{matrix factorization} $X$  is visualized as follows
 \[\xymatrix{
  X^0\ar@<+.7ex>[rr]^-{d_X^0} && X^1 \ar@<+.7ex>[ll]^-{d_X^1},
}\]
with each component $X^i$ a free $S$-module of finite rank satisfying
$$d_X^1\circ d_X^0=\omega {\rm Id}_{X^0} \mbox{ and } d_X^0\circ d_X^1=\omega {\rm Id}_{X^1}.$$
An $R$-module $M$ is MCM if and only if there is a matrix factorization $X$ such that $M$ is isomorphic to ${\rm Coker}(d_X^0)$.  We mention that matrix factorizations appear in link homology \cite{KR},  Hodge theory \cite{BFK, KKP} and mathematical physics \cite{Car, KapL}.

For the general case, \emph{higher matrix factorizations} associated to a regular sequence are introduced in \cite{EP16}. Analogously, an $R$-module $M$ is MCM if and only if there is a higher matrix factorization $Z$ such that $M$ is isomorphic to a certain cokernel related to $Z$; compare \cite{EP}.

However, the category of higher matrix factorizations does not have an obvious exact structure. Consequently, its homotopy category has no obvious triangulated structure.

In view of the hypersurface case \cite{Eis}, the following problem is fundamental.
\vskip 5pt

\noindent  {\bf Problem}. \emph{For a complete intersection $R=S/{(\omega_1, \cdots, \omega_n)}$, how can one describe the stable category $\underline{\rm MCM}(R)$ in terms of free $S$-modules of finite rank?  }

\vskip 5pt
We mention that Problem is analogous to \cite[Question~1]{BW}. The latter leads to a description of $\underline{\rm MCM}(R)$ in terms of matrix factorizations \cite{Orl2, BW} over a certain non-affine scheme, which relies on \cite{Orl06}.

In Theorem~B below, we solve Problem completely. The main tools are \emph{higher-dimensional matrix factorizations} of  a regular sequence. The classical matrix factorizations will be viewed as $1$-dimensional objects.

To be more precise, an $n$-dimensional matrix factorization of  a regular sequence $(\omega_1, \cdots, \omega_n)$ is an $n$-dimensional commutative cube $X=(X^{\alpha})_{\alpha\in \{0, 1\}^n}$ consisting of free $S$-modules $X^\alpha$ of finite rank such that any edge of $X$ in direction $i$ belongs to $\mathbf{MF}(S; \omega_i)$.  For example,  a $2$-dimensional matrix factorization of $(\omega_1, \omega_2)$ is a  commutative square with rows in $\mathbf{MF}(S; \omega_1)$ and columns in $\mathbf{MF}(S; \omega_2)$; see (\ref{diag:fact}) in Subsection~\ref{subsec:6.1}.
\[
\xymatrix{
 X^{0, 1}\ar@<+.7ex>[rr]^-{d_1^{0, 1}}    \ar@<+.7ex>[dd]^-{d_2^{0,1}} && X^{1, 1} \ar@<+.7ex>[ll]^-{d_1^{1,1}}  \ar@<+.7ex>[dd]^-{d_2^{1,1}}\\ \\
 X^{0, 0}\ar@<+.7ex>[rr]^-{d_1^{0, 0}} \ar@<+.7ex>[uu]^-{d_2^{0, 0}}  && X^{1, 0}  \ar@<+.7ex>[uu]^-{d_2^{1, 0}} \ar@<+.7ex>[ll]^-{d_1^{1,0}}
}\]
We denote by $\mathbf{MF}(S; \omega_1, \cdots, \omega_n)$ the category of $n$-dimensional matrix factorizations. It is naturally a Frobenius exact category. Thus, the stable category $\underline{\mathbf{MF}}(S; \omega_1, \cdots, \omega_n)$ is triangulated \cite{Hap}.  The corresponding \emph{p-null-homotopical morphisms} are described in Subsection~\ref{subsec:9.1}.

Set ${\bf 1}=(1, 1, \cdots, 1)\in \{0, 1\}^n$. For each $1\leq i\leq n$,  we denote by $\epsilon_i$ the corresponding unit vector in $\{0, 1\}^n$. For an $n$-dimensional matrix factorization $X$, we consider the following map between free $S$-modules
$$\sum_{1\leq i \leq n}d_i^{{\bf 1}-\epsilon_i} \colon \bigoplus_{1\leq i\leq n} X^{{\bf 1}-\epsilon_i} \longrightarrow X^{\bf 1},$$
whose cokernel is denoted by ${\rm TCok}(X)$, called the \emph{total-cokernel} of $X$. Since each $\omega_i$ vanishes on ${\rm TCok}(X)$, it is naturally an $R$-module.

The following observation justifies the new notion of higher-dimensional matrix factorizations; see Remark~\ref{rem:PropA}.
\vskip 5pt

\noindent {\bf Proposition~A.} \; \emph{Let $R=S/{(\omega_1, \cdots, \omega_n)}$ be a complete intersection. Then an $R$-module  is {\rm MCM} if and only if it is isomorphic to the total-cokernel of some  $n$-dimensional matrix factorization.}

\vskip 5pt
Consequently, the following functor is
$${\rm TCok}\colon \mathbf{MF}(S; \omega_1, \cdots, \omega_n)\longrightarrow  {\rm MCM}(R)$$
well defined. To describe the category $\underline{\rm MCM}(R)$, we have to introduce a very subtle subcategory of  $\mathbf{MF}(S; \omega_1, \cdots, \omega_n)$.

We define the full subcategory $\mathbf{MF}^{\bf 0}(S; \omega_1, \cdots, \omega_n)$ inductively as follows: an $n$-dimensional matrix factorization $X$ belongs to $\mathbf{MF}^{\bf 0}(S; \omega_1, \cdots, \omega_n)$ if and only if each facet of $X$ containing $X^{\bf 1}$ is a projective object in the relevant category of $(n-1)$-dimensional matrix factorizations. This category is still a Frobenius exact category, whose  stable category is denoted by $\underline{\mathbf{MF}}^{\bf 0}(S; \omega_1, \cdots, \omega_n)$.

\vskip 5pt

\noindent {\bf Theorem~B (= Theorem~\ref{thm:B}).}  \; \emph{Let $R=S/{(\omega_1, \cdots, \omega_n)}$ be a complete intersection. Then the assignment $X\mapsto {\rm TCok}(X)$ induces a triangle equivalence
$$\underline{\mathbf{MF}}^{\bf 0}(S; \omega_1, \cdots, \omega_n)\simeq \underline{\rm MCM}(R).$$}

\vskip 5pt

In what follows, we compare our higher-dimensional matrix factorizations with higher matrix factorizations in \cite{EP16}. For details, we refer to Subsection~\ref{subsec:9.1}.

(i) Any $n$-dimensional matrix factorization $X$ yields a higher matrix factorization,  whose underlying modules are $(\bigoplus_{1\leq i\leq n} X^{{\bf 1}-\epsilon_i} , X^{\bf 1})$. Consequently,  Proposition~A gives another proof of the result that any MCM $R$-module is  a higher matrix factorization module; see \cite[Definition~1.2.2]{EP16} and \cite[Theorem~10.5]{EP}.

(ii) Higher matrix factorizations in \cite{EP16} are essential tools to obtain minimal free resolutions of $R$-modules; compare \cite{EP}. Unfortunately, higher-dimesnional matrix factorizations seemingly do not play such a role.

(iii) The category of higher-dimensional matrix factorizations is naturally a Frobenius exact category, which makes Theorem~B possible. However, it seems that the category of higher matrix factorizations in \cite{EP16} does not have such a property.

(iv) Higher-dimensional matrix factorizations extend well to higher-dimensional module factorizations, which work well in the noncommutative setting. However, it is not clear how higher matrix factorizations in \cite{EP16} extend to or  work for noncommutative rings.

\vskip 5pt

To explain (iv) in more details, we mention that matrix factorizations over noncommutative rings are introduced in \cite{CCKM}; see also \cite{MU}. As their natural extensions, ($1$-dimensional) \emph{module factorizations} are studied in \cite{Chen24}.  In the commutative case, module factorizations  are  called linear factorizations in \cite{DM} and generalized matrix factorizations in \cite{EP}. The main concern here is \emph{higher-dimensional module factorizations} with Gorenstein projective components; compare \cite{BFNS}. Here, we recall that finitely generated Gorenstein projective modules are introduced in \cite{ABr}, which are natural extensions of MCM modules over  commutative Gorenstein local rings. For general Gorenstein projective modules, we refer to \cite{EJ}. We mention a nice survey \cite{CFH}.

\vskip 5pt

\subsection{The main result}  Let us describe the main result in full generality. Let $A$ be any ring. By a \emph{regular sequence} in $A$, we mean a sequence $(\omega_1, \cdots, \omega_n)$ of elements such that $\omega_1$ is regular and normal in $A$ and that $\bar{\omega_i}$ is regular and normal in the quotient ring $A/{(\omega_1, \cdots, \omega_{i-1})}$ for each $2\leq i\leq n$; see \cite{KKZ}.  Set $B=A/{(\omega_1, \cdots, \omega_n)}$.

For a technical reason, we have to strengthen the notion. For this end, we introduce a \emph{type} $(\sigma_1, \cdots, \sigma_n; \xi_{ij})$, which consists of   automorphisms  $\sigma_i$ on $A$ and certain invertible elements $\xi_{ij}$ in $A$; see Definition~\ref{defn:type}. Furthermore, we introduce the notion of a \emph{regular sequence of type} $(\sigma_1, \cdots, \sigma_n; \xi_{ij})$; see Definition~\ref{defn:type-n}. Associated to these data, we introduce the abelian category $\mathbf{F}(A; \omega_1, \cdots, \omega_n)$ of \emph{$n$-dimensional module factorizations}, which are commutative cubes with edges being module factorizations \cite{Chen24}.

For any ring $R$, we denote by $R\mbox{-GProj}$ the category of Gorenstein projective $R$-modules and by $R\mbox{-Gproj}$ the category of finitely generated ones. Their stable categories are denoted by $R\mbox{-\underline{GProj}}$  and $R\mbox{-\underline{Gproj}}$, respectively.

We consider the full subcategory $\mathbf{GF}(A; \omega_1, \cdots, \omega_n)$ consisting of $n$-dimensional module factorizations $X=(X^{\alpha})_{\alpha\in \{0, 1\}^n}$ with Gorenstein projective components. We still have the total-cokernel functor
$${\rm TCok}\colon \mathbf{GF}(A; \omega_1, \cdots, \omega_n)\longrightarrow B\mbox{-GProj}.$$
Similar to the consideration above, we introduce the full subcategory $\mathbf{GF}^{\bf 0}(A; \omega_1, \cdots, \omega_n)$ of $\mathbf{GF}(A; \omega_1, \cdots, \omega_n)$  formed by those factorizations $X$, each of  whose facets containing $X^{\bf 1}$ is a projective object in the relevant category of  $(n-1)$-dimensional module factorizations. In particular, all components of $X$, possibly except $X^{\bf 0}$, are projective $A$-modules.

An object $X$ in  $\mathbf{GF}(A; \omega_1, \cdots, \omega_n)$  is called an \emph{$n$-dimensional matrix factorization}, if each component $X^\alpha$ is a finitely generated projective $A$-module. These objects form the subcategory $\mathbf{MF}(A; \omega_1, \cdots, \omega_n)$. Similarly, we have the category   $\mathbf{MF}^{\bf 0}(A; \omega_1, \cdots, \omega_n)$.

The main result is a vast extension and generalization of Theorem~B.

\vskip 5pt

\noindent {\bf Theorem~C (= Theorem~\ref{thm:general}).}\; \emph{Let  $(\omega_1, \cdots, \omega_n)$ be a regular sequence of type $(\sigma_1, \cdots, \sigma_n; \xi_{ij})$ in $A$. Set $B= A/{(\omega_1, \cdots, \omega_n)}$.  Then the total-cokernel functor induces a triangle equivalence
$${\rm TCok}\colon \underline{\mathbf{GF}}^{\bf 0}(A; \omega_1, \cdots, \omega_n)\stackrel{\sim}\longrightarrow B\mbox{-\underline{\rm GProj}}.$$
When $A$ is left noetherian, we have a restricted triangle equivalence
$${\rm TCok}\colon \underline{\mathbf{MF}}^{\bf 0}(A; \omega_1, \cdots, \omega_n)\stackrel{\sim}\longrightarrow B\mbox{-\underline{\rm Gproj}}^{<+\infty}.$$}

\vskip 5pt

Here, $ B\mbox{-Gproj}^{<+\infty}$ denotes the full  subcategory of $B\mbox{-Gproj}$ consisting modules whose underlying $A$-modules have  finite projective dimension.  When $A$ is left noetherian, we will see that the restricted total-cokernel functor
$${\rm TCok}\colon \mathbf{MF}^{\bf 0}(A; \omega_1, \cdots, \omega_n) \longrightarrow B\mbox{-Gproj}^{<+\infty}$$ is dense; see Remarks~\ref{rem:dense-n} and \ref{rem:PropA}. In particular, we deduce Proposition~A.

The proof of Theorem~C relies heavily on \cite{Chen24}. We emphasize that even if the ring $A$ is commutative, we still have to use module factorizations over noncommutative rings, namely, the \emph{twisted matrix rings}; see Section~\ref{sec:modf}.

Theorem~C applies well to the case when $A$ is the quantum polynomial algebra; see Theorem~\ref{thm:qci}. The corresponding quotient algebra $B$ might be chosen to be the quantum complete intersection \cite{AGP}; see also \cite{BE, BO}.

\vskip 5pt

The paper is structured as follows. Section~\ref{sec:2} studies the category of factorizations in an arbitrary category. We recall basic facts on exact categories and their stable categories in Section~\ref{sec:3}. We study the canonical exact structure on the category of factorizations in Section~\ref{sec:4}. In Section~\ref{sec:modf}, we study module factorizations and GP-compatible idempotents.

In Section~\ref{sec:2-dim}, we introduce two-dimensional factorizations with respect to two commuting natural transformations.  In Section~\ref{sec:7}, we prove   Theorem~C in the first new case $n=2$; see Theorem~\ref{thm:2-dim}. Using induction and the twisted matrix ring,  we prove Theorem~C in the full generality in Section~\ref{sec:8}. We illustrate the main result for complete intersections in the final section.

By default, rings mean unital rings, and modules mean left unital modules. We denote by $\Sigma$ the suspension functor of any triangulated category. For triangulated categories, we refer to \cite{BBD, Hap}.

\section{The categories of factorizations}  \label{sec:2}

In this section, we study the category of factorizations and introduce  several related adjoint pairs.

Let $\mathcal{C}$ be any category with an autoequivalence $T\colon \mathcal{C}\rightarrow \mathcal{C}$. Fix a natural transformation $\omega\colon {\rm Id}_\mathcal{C}\rightarrow T$, which satisfies $\omega T=T\omega$, or equivalently, $\omega_{T(C)}=T(\omega_C)$ for any object $C$.

By an \emph{$\omega$-factorization} \cite{BDFIK}, we mean a quadruple
$$X=(X^0, X^1; d_X^0, d_X^1)$$
consisting of two objects $X^0,  X^1$, and two morphisms $d_X^0\colon X^0\rightarrow X^1$, $d_X^1\colon X^1\rightarrow T(X^0)$, which are subject to the following conditions:
$$d_X^1\circ d_X^0=\omega_{X^0}  \mbox{ and  } T(d_X^0)\circ d_X^1=\omega_{X^1}.$$
An $\omega$-factorization might be visualized as follows:
\[\xymatrix{
X^0 \ar[r]^-{d_X^0}  & X^1 \ar[r]^-{d_X^1} & T(X^0)  \mbox{ or }  X^0\ar@<+.7ex>[rr]^-{d_X^0} && X^1. \ar@<+.7ex>@{~>}[ll]^-{d_X^1}
}\]
Here, the curved arrow from $X^1$ to $X^0$ means the morphism $d_X^1\colon X^1\rightarrow T(X^0)$.

Let $Y=(Y^0, Y^1; d_Y^0, d_Y^1)$ be another $\omega$-factorization. A morphism $f=(f^0, f^1)\colon X\rightarrow Y$ between these two $\omega$-factorizations consists of two morphisms $f^0\colon X^0\rightarrow Y^0$ and $f^1\colon X^1\rightarrow Y^1$ in $\mathcal{C}$ satisfying
$$d_Y^0\circ f^0=f^1\circ d_X^0 \mbox{ and } d_Y^1\circ f^1=T(f^0)\circ d_X^1.$$
The composition of morphisms is defined componentwise. This gives rise to the category $\mathbf{F}(\mathcal{C}; \omega)$ of $\omega$-factorizations; compare \cite{BDFIK, BJ} and \cite[Remark~2.7]{Posi}.

 Take a quasi-inverse $T^{-1}$ of $T$, which fits into an adjoint pair $(T^{-1}, T)$.  We have the unit $\eta\colon {\rm Id}_\mathcal{C}\rightarrow TT^{-1}$ and the counit $\varepsilon\colon T^{-1}T\rightarrow {\rm Id}_\mathcal{C}$.  The adjunction quadruple
 \begin{align}\label{quad:1}
     (T^{-1}, T;  \eta, \varepsilon)
 \end{align}
  will be fixed.

To some extent, the following remark enables us to view $T$ as an automorphism of the category $\mathcal{C}$; compare \cite[2.1]{KV}.

\begin{rem}\label{rem:nat-trans}
Associated to $\omega \colon {\rm Id}_\mathcal{C}\rightarrow T$, we  define a natural transformation
 $$\omega^{(-1)}=\varepsilon\circ T^{-1}\omega\colon T^{-1}\longrightarrow {\rm Id}_\mathcal{C}.$$
We observe
\begin{align}\label{equ:omega1}
    T\omega^{(-1)}\circ \eta= T\varepsilon \circ TT^{-1}\omega \circ \eta=(\eta T)^{-1}\circ TT^{-1}\omega \circ \eta=\omega.
\end{align}
Here, the middle equality uses $T\varepsilon=(\eta T)^{-1}$,  and the rightmost one uses the naturalness of $\eta$. Moreover, we have
\begin{align}\label{equ:omega11}
\omega\circ \varepsilon= \varepsilon T\circ T^{-1}T\omega= \varepsilon T\circ T^{-1}\omega T=\omega^{(-1)}T,
\end{align}
where the leftmost equality uses the naturalness of $\varepsilon$.

We claim that the following identity holds.
\begin{align}\label{equ:omega111}
\varepsilon T^{-1}\circ T^{-1}\omega T^{-1}= T^{-1}\varepsilon \circ T^{-2}\omega
\end{align}
For this end, we observe that by using the naturalness of $\varepsilon$ twice, we have the following commutative diagram.
\[\xymatrix{
T^{-2} \ar[rr]^-{T^{-1}\omega T^{-1}} && T^{-1}TT^{-1} && T^{-1} \ar[ll]_-{T^{-1}\eta}\\
T^{-2}TT^{-1}\ar[u]^-{T^{-1}\varepsilon T^{-1}} \ar[rr]^-{T^{-2}T\omega T^{-1}}  && T^{-2}T^2T^{-1} \ar[u]^-{T^{-1}\varepsilon TT^{-1}} && T^{-2}T \ar[u]_-{T^{-1}\varepsilon} \ar[ll]_-{T^{-2}T\eta}
}\]
We have the following identity.
\begin{align*}
\varepsilon T^{-1}\circ T^{-1}\omega T^{-1} & =(T^{-1}\eta)^{-1}\circ T^{-1}\omega T^{-1}\\
                      &=T^{-1}\varepsilon \circ (T^{-2}T\eta)^{-1}\circ T^{-2}T\omega T^{-1}\circ (T^{-1}\varepsilon T^{-1})^{-1}\\
                      &=T^{-1}\varepsilon \circ (T^{-2}T\eta)^{-1}\circ T^{-2}\omega T T^{-1}\circ T^{-2}\eta\\
                      &=T^{-1}\varepsilon\circ T^{-2}\omega
\end{align*}
Here, the first  equality uses $\varepsilon T^{-1}=(T^{-1}\eta)^{-1}$, the second one uses the commutative diagram above, the third one uses $\omega T=T\omega$ and $\varepsilon T^{-1}=(T^{-1}\eta)^{-1}$, and the final one uses $(T\eta)^{-1}\circ \omega TT^{-1}\circ \eta=\omega$, which is a consequence of the naturalness of $\omega$.

Rewriting (\ref{equ:omega111}), we have
$$\omega^{(-1)}T^{-1}=T^{-1}\omega^{(-1)}.$$
Combining (\ref{equ:omega111}) with $\varepsilon T^{-1}=(T^{-1}\eta)^{-1}$, we obtain
$$T^{-1}\omega T^{-1}= T^{-1}\eta \circ T^{-1}\varepsilon\circ T^{-2}\omega.$$
Since $T^{-1}$ is  fully faithful, we infer that
\begin{align}\label{equ:omega2}
\omega T^{-1}=\eta\circ \varepsilon \circ T^{-1}\omega=\eta\circ \omega^{(-1)}.
\end{align}
\end{rem}

\vskip 3pt

Let $C$ be an object in $\mathcal{C}$. We have two \emph{trivial $\omega$-factorizations}
$$\theta^0(C)=(C, C; {\rm Id}_C, \omega_C) \mbox{ and }\theta^1(C)=(T^{-1}(C), C; \omega^{(-1)}_C, \eta_C).$$
Here, to verify that $\theta^1(C)$ is an $\omega$-factorization, one uses (\ref{equ:omega1}) and (\ref{equ:omega2}). This gives rise to two functors
\begin{align}\label{theta:1}
    \theta^s\colon \mathcal{C}\longrightarrow \mathbf{F}(\mathcal{C}; \omega)
\end{align}
for $s=0, 1$.

For $s=0,1$, we have the \emph{projection functors}
$${\rm pr}^s\colon \mathbf{F}(\mathcal{C}; \omega)\longrightarrow \mathcal{C},$$
which send $X$ to its $s$-th component $X^s$. For each  $\omega$-factorization $X$, we have the \emph{shifted $\omega$-factorization}
$$S(X)=(X^1, T(X^0); d_X^1, T(d_X^0)).$$
This gives rise to an autoequivalence of categories
$$S\colon \mathbf{F}(\mathcal{C}; \omega)\longrightarrow \mathbf{F}(\mathcal{C}; \omega).$$
We observe  two equalities
$${\rm pr}^1 S=T{\rm pr}^0 \mbox{ and } {\rm pr}^0 S={\rm pr}^1,$$
and two natural isomorphisms
$$S\theta^1 \simeq \theta^0\mbox{ and } \theta^1T\simeq S\theta^0.$$
For the first isomorphism, we use (\ref{equ:omega1}), and for the second one, we use (\ref{equ:omega11}).

The following results are analogous to \cite[Lemma~3.2]{Chen24}.

\begin{lem}\label{lem:3adj}
Keep the notation as above. Then the following statements hold.
\begin{itemize}
    \item[(1)] We have two adjoint pairs $(\theta^0, {\rm pr}^0)$ and $({\rm pr}^0, S\theta^0)$.
    \item[(2)] We have two adjoint pairs $(\theta^1, {\rm pr}^1)$ and $({\rm pr}^1, \theta^0)$.
    \item[(3)] We have an adjoint pair $({\rm pr}^1 S, \theta^1)$.
\end{itemize}
\end{lem}

\begin{proof}
 (1) For each object $C$ in $\mathcal{C}$ and any $\omega$-factorization $X=(X^0, X^1; d_X^0, d_X^1)$, we have a natural isomorphism
\begin{align}\label{iso:theta0}
{\rm Hom}_\mathcal{C}(C, X^0)\longrightarrow {\rm Hom}_{\mathbf{F}(\mathcal{C}; \omega)}(\theta^0(C), X), \quad g \mapsto (g, d_X^0\circ g).
\end{align}
This yields the first adjoint pair. For the second one, we use the following isomorphism
$$ {\rm Hom}_\mathcal{C}(X^0, C)\longrightarrow {\rm Hom}_{\mathbf{F}(\mathcal{C}; \omega)}(X, S\theta^0(C)),  \quad g \mapsto (g, T(g)\circ d_X^1).$$

(2) The adjunctions follow by combining the observation above and (1). Alternatively,  the first adjoint pair follows from the following isomorphism
\begin{align}\label{iso:theta1}
{\rm Hom}_\mathcal{C}(C, X^1)\longrightarrow {\rm Hom}_{\mathbf{F}(\mathcal{C}; \omega)}(\theta^1(C), X), \quad g\mapsto (\varepsilon_{X^0}\circ T^{-1}(d_X^1\circ g),g).
\end{align}
The second one uses the following isomorphism
\begin{align}\label{iso:theta01}
    {\rm Hom}_\mathcal{C}(X^1, C)\longrightarrow {\rm Hom}_{\mathbf{F}(\mathcal{C}; \omega)}(X, \theta^0(C)), \quad g\mapsto (g\circ d_X^0,g).
\end{align}

(3) This adjoint pair follows from the isomorphism $S\theta^1 \simeq \theta^0$ and the adjoint pair $({\rm pr}^1, \theta^0)$. Alternatively, we have a natural isomorphism
\begin{align}\label{iso:theta11}
   {\rm Hom}_\mathcal{C}(T(X^0), C)\longrightarrow {\rm Hom}_{\mathbf{F}(\mathcal{C}; \omega)}(X, \theta^1(C)), \quad g\mapsto (T^{-1}(g)\circ (\varepsilon_{X^0})^{-1}, g\circ d_X^1).
\end{align}
Here, to verify that $(T^{-1}(g)\circ (\varepsilon_{X^0})^{-1}, g\circ d_X^1)\colon X \rightarrow \theta^1(C)$ is a morphism, one uses $\omega^{(-1)}T\circ \varepsilon^{-1}=\omega$; see (\ref{equ:omega11}).
\end{proof}

\section{Preliminaries on exact categories} \label{sec:3}

In this section, we collect some basic facts and notions on exact
categories and Frobenius categories. The basic references
are \cite[Appendix~A]{Kel90} and \cite{Bu10}.

\vskip 5pt

Let $\mathcal{A}$ be an additive category. A \emph{composable pair}
of morphisms is a sequence
$$X\stackrel{i} \longrightarrow Y \stackrel{d}
\longrightarrow Z,$$
which is denoted by $(i, d)$. Two
composable pairs $(i, d)$ and $(i', d')$ are \emph{isomorphic}
provided that there are isomorphisms $f\colon X\rightarrow X'$,
$g\colon Y\rightarrow Y'$ and $h\colon Z\rightarrow Z'$ such that
$$g\circ i=i'\circ f \mbox{ and } h\circ d=d' \circ g.$$
A composable pair $(i, d)$ is called a \emph{kernel-cokernel pair}  provided that
$i$ is a kernel of $d$ and that $d$ is a cokernel of $i$.

 An \emph{exact structure} on $\mathcal{A}$
 is a chosen class $\mathcal{E}$ of  kernel-cokernel pairs in
 $\mathcal{A}$, which is closed under isomorphisms and
is subject to the following axioms (Ex0), (Ex1), (Ex1)$^{\rm op}$,
(Ex2) and (Ex2)$^{\rm op}$. Any pair $(i, d)$ in the exact structure
$\mathcal{E}$ is called a \emph{conflation}, while $i$ is called an
\emph{inflation} and $d$ is called a \emph{deflation}. The pair
$(\mathcal{A}, \mathcal{E})$ is called an \emph{exact category} \cite{Qui73}.
When the exact structure $\mathcal{E}$ is understood, we will  simply say that $\mathcal{A}$ is an exact category.

Following \cite[Appendix~A]{Kel90}, the axioms of  an exact category are
listed as follows:

\begin{enumerate}
\item[(Ex0) \;  ] the identity morphism of the zero object is a deflation;

\item[(Ex1) \;  ] the composition of any  two deflations is a deflation;

\item[(Ex1)$^{\rm op}$] the composition of any two inflations is an
inflation;

\item[(Ex2) \;  ] for a deflation $d\colon Y \rightarrow Z$ and a
morphism $f\colon Z'\rightarrow Z$, there exists a pullback diagram in $\mathcal{A}$
such that $d'$ is a deflation:
\[\xymatrix{  Y' \ar@{.>}[r]^{d'} \ar@{.>}[d]_-{f'} & Z' \ar[d]^{f} \\
   Y  \ar[r]^-{d} & Z}\]

\item[(Ex2)$^{\rm op}$] for an inflation $i\colon X \rightarrow Y$ and a
morphism $f\colon X\rightarrow X'$, there exists a  pushout diagram in $\mathcal{A}$
such that $i'$ is an inflation:
\[\xymatrix{  X \ar[r]^-{i} \ar[d]_-{f} & Y \ar@{.>}[d]^-{f'} \\
   X'  \ar@{.>}[r]^-{i'} & Y'}\]
\end{enumerate}

We remark that the axiom (Ex1)$^{\rm op}$ can be deduced from
the remaining ones; see \cite[Appendix~A]{Kel90}.

 For an exact category $\mathcal{A}$, a full additive subcategory
$\mathcal{B}\subseteq \mathcal{A}$ is said to be
\emph{extension-closed} provided that for any conflation
$X\stackrel{i} \rightarrow Y \stackrel{d}\rightarrow Z$ with $X,
Z\in \mathcal{B}$ we necessarily have $Y\in \mathcal{B}$. In this case, the
subcategory $\mathcal{B}$ inherits the exact structure from
$\mathcal{A}$ and becomes an exact category.

\begin{exm}
  {\rm  Any abelian category  has a canonical exact structure such that
conflations are induced by short exact sequences. Consequently, any
 extension-closed subcategory in an abelian category
has a \emph{canonical exact structure} and becomes an
exact category.}
\end{exm}

Let $\mathcal{A}$ be an exact category. Recall that an object $P$ in $\mathcal{A}$ is \emph{projective}
provided that the functor ${\rm Hom}_\mathcal{A}(P, -)$ sends
conflations to short exact sequences of abelian groups; this is equivalent to that any
deflation ending at $P$ splits. The exact category $\mathcal{A}$ is
said to \emph{have enough projective objects}, provided that each
object $X$ fits into a deflation $d\colon P\rightarrow X$ with $P$
projective. Dually, one defines \emph{injective object}
and \emph{having enough injective objects}.

Assume that $\mathcal{A}$ has enough projective objects. The (projectively) stable category $\underline{\mathcal{A}}$ of $\mathcal{A}$ is defined as follows. It has the same objects as $\mathcal{A}$. For two objects $X$ and $Y$, its Hom group, denoted by $\underline{\rm Hom}_\mathcal{A}(X, Y)$, is defined to be the quotient of ${\rm Hom}_\mathcal{A}(X, Y)$ modulo  the subgroup formed by those morphisms that factor through projective objects.

An exact category $\mathcal{A}$ is said to be \emph{Frobenius},
provided that it has enough projective and enough injective objects,
 and the class of projective objects coincides with the class of
injective objects; see \cite[Section~3]{Hel60}. The importance of
Frobenius categories lies in  the following fundamental result: the stable category of a Frobenius category has a canonical triangulated structure; see \cite[I.2]{Hap} and \cite[1.2]{Kel90}.

Recall that an additive functor $F\colon \mathcal{A}\rightarrow
\mathcal{A}'$ between two exact categories is called \emph{exact}
provided that it sends conflations to conflations; an exact functor
$F\colon \mathcal{A}\rightarrow \mathcal{A}'$ is said to be an
\emph{equivalence of exact categories} provided that $F$ is an
equivalence and there exists a quasi-inverse of $F$ which is also exact.

Let $\mathcal{A}$ and $\mathcal{A}'$ be two Frobenius categories. Assume that $F\colon \mathcal{A}\rightarrow
\mathcal{A}'$ is an exact functor which sends projective objects to projective objects. Then it induces a triangle functor $F\colon \underline{\mathcal{A}}\rightarrow \underline{\mathcal{A}}'$ between their stable categories.

Let $A$ be any  ring. Denote by $A\mbox{-Mod}$ the abelian category of left $A$-modules, and by $A\mbox{-Proj}$ it full subcategory formed by projective modules. Furthermore, $A\mbox{-mod}$ and $A\mbox{-proj}$  denote the full subcategory formed by finitely generated $A$-modules and finitely generated projective $A$-modules, respectively.

The main examples of our concern are as follows.

\begin{exm}
    {\rm A unbounded complex $P^\bullet=(P^n, d_P^n)_{n\in \mathbb{Z}}$ of projective $A$-modules is called \emph{totally acyclic} if it is acyclic and for any projective $A$-modules $Q$, the Hom complex ${\rm Hom}_A(P^\bullet, Q)$ is also acyclic. An $A$-modules $G$ is called \emph{Gorenstein projective} \cite{EJ} if there is a totally acyclic complex $P^\bullet$ such that $G$ is isomorphic to $Z^1(P^\bullet)$, the first cocycle of $P^\bullet$. We observe that projective modules are Gorenstein projective.

    Denote by $A\mbox{-GProj}$ the category formed by Gorenstein projective modules.  By \cite[Theorem~2.5]{Holm} or \cite[Lemma~2.3]{AM}, the subcategory  $A\mbox{-GProj}$ of  $A\mbox{-Mod}$  is closed under extensions, and becomes an exact category. Moreover, it is Frobenius, whose projective-injective objects are precisely all projective $A$-modules; see \cite[Proposition~3.8(i)]{Bel}. Consequently, the stable category $A\mbox{-\underline{GProj}}$ is canonically triangulated.

    When the ring $A$ is left noetherian, we denote by $A\mbox{-Gproj}$ the category of finitely generated Gorenstein projective modules. It is also a Frobenius exact category, whose projective objects are precisely finitely generated projective $A$-modules. }
\end{exm}

Let $e=e^2$ be an idempotent in $A$. Then $eAe$ is a ring with $e$ its unit. We consider the Schur functor
$$S_e\colon A\mbox{-Mod}\longrightarrow eAe\mbox{-Mod}, \; M\mapsto eM.$$

\begin{defn}
The idempotent $e$ is said to be \emph{GP-compatible}, provided that the left $eAe$-module $eA$ is projective and that for each Gorenstein projective $A$-modules $G$, the $eAe$-module $eG$ is Gorenstein projective.  \hfill $\square$
\end{defn}

Following \cite{DI}, we set
$$A\mbox{-GProj}^e=\{G\in A\mbox{-GProj}\; |\; \mbox{ the } eAe\mbox{-module } eG \mbox{ is projective}\},$$
which is an extension-closed subcategory of $A\mbox{-GProj}$ and becomes an exact category. When $e$ is GP-compatible, the exact category $A\mbox{-GProj}^e$ is Frobenius. The stable category $A\mbox{-\underline{GProj}}^e$ is the essential kernel of the \emph{induced Schur functor}
$$S_e\colon A\mbox{-\underline{GProj}}\longrightarrow eAe\mbox{-\underline{GProj}}.$$
We refer to Proposition~\ref{prop:M-2} for concrete examples on GP-compatible idempotents.

Let us recall a recollement \cite[1.4]{BBD} situation between stable categories.   Assume that $\pi\colon \mathcal{B} \rightarrow \mathcal{A}$ is an exact functor between two Frobenius categories,  which sends projective objects to projective objects. We assume further that $\pi$ has a left adjoint $\pi_\lambda$ and a right adjoint $\pi_\rho$, both of which send projective objects to projective objects. We will require that both $\pi_\lambda$ and $\pi_\rho$ are fully faithful.

We have the induced triangle functors
$$\pi\colon \underline{\mathcal{B}}\longrightarrow \underline{\mathcal{A}}, \;  \pi_\lambda \colon \underline{\mathcal{A}}\longrightarrow \underline{\mathcal{B}} \mbox{ and }  \pi_\rho \colon \underline{\mathcal{A}}\longrightarrow \underline{\mathcal{B}}. $$
Set $\mathcal{K}=\{B\in \mathcal{B}\; |\; \pi(B) \mbox{ is projective in }\mathcal{A}\}$; it is also a Frobenius exact category. Moreover, its stable category $\underline{\mathcal{K}}$ coincides with the essential kernel of $\pi\colon \underline{\mathcal{B}}\rightarrow \underline{\mathcal{A}}$.

The following result is well known; see  \cite[Lemmas~2.3 and 2.4]{Chen-VB}.

\begin{lem} \label{lem:rec}
    Keep the assumptions above. Then we have an induced recollement.
    \[\xymatrix{
  \underline{\mathcal{K}}\;\ar[rr]|-{\rm inc} &&\; \underline{\mathcal{B}}\; \ar[rr]|-{\pi}
  \ar@/^1.5pc/[ll]|{{\rm inc}_\rho}  \ar@/_1.5pc/[ll]|{{\rm inc}_\lambda} &&
  \; \underline{\mathcal{A}}  \ar@/^1.5pc/[ll]|{\pi_\rho}  \ar@/_1.5pc/[ll]|{\pi_\lambda}
}\]
    Here, ``${\rm inc}$" denotes the inclusion functor. Moreover, the functors ${\rm inc}_\lambda$ and ${\rm inc}_\rho$ are determined by the following  exact triangles in $\underline{\mathcal{B}}$
    $$\pi_\lambda \pi(X)\rightarrow X\rightarrow {\rm inc}_\lambda(X)\rightarrow \Sigma\pi_\lambda \pi(X) \mbox{ and } {\rm inc}_\rho(X)\rightarrow X\rightarrow \pi_\rho \pi(X)\rightarrow \Sigma {\rm inc}_\rho(X).$$
\end{lem}

\section{Exact structures on factorizations}\label{sec:4}

 In this section, we study the canonical exact structure on the category of factorizations.  We prove that the category of factorizations is Frobenius, provided that so is the category of underlying  objects; see Proposition~\ref{prop:Frobenius}.

 Let $\mathcal{A}$ be an additive category. Fix an autoequivalence $T$ on $\mathcal{A}$ and a natural transformation $\omega\colon {\rm Id}_\mathcal{A}\rightarrow T$ satisfying $\omega T=T\omega$. The category $\mathbf{F}(\mathcal{A}; \omega)$ of $\omega$-factorizations is additive.

The following fact is elementary.

\begin{lem}\label{lem:kerinF}
    Let $f=(f^0, f^1)\colon X\rightarrow Y$ be a morphism in $\mathbf{F}(\mathcal{A}; \omega)$. Assume that both $f^0$ and $f^1$ have kernels in $\mathcal{A}$. Then the morphism $f$ has a kernel, which is given componentwise.
\end{lem}

\begin{proof}
Assume that $i^s\colon K^s\rightarrow X^s$ is  a kernel of $f^s$ for $s=0,1$. There is a unique morphism $d_K^0\colon K^0\rightarrow K^1$ satisfying $i^1\circ d_K^0=d_X^0\circ i^0$. Since $T(i^0)$ is a kernel of $T(f^0)$,  there exists a unique morphism $d_K^1\colon K^1\rightarrow T(K^0)$ satisfying $T(i^0)\circ d_K^1=d_X^1\circ i^1$. It is rountine to verify that $K=(K^0, K^1;d_K^0, d_K^1)$ is an $\omega$-factorization, and that the morphism $i=(i^0, i^1)\colon K\rightarrow X$ is a required kernel.
\end{proof}

We assume further that $(\mathcal{A},\mathcal{E})$ is an exact category. Denote by $\tilde{\mathcal{E}}$ the class of composable pairs $(i, p)$ in $\mathbf{F}(\mathcal{A}; \omega)$ such that both $(i^0,p^0)$ and $(i^1, p^1)$ belong to $\mathcal{E}$. By Lemma~\ref{lem:kerinF} and its dual, any element in  $\tilde{\mathcal{E}}$ is a kernel-cokernel pair in $\mathbf{F}(\mathcal{A}; \omega)$.

The following result is analogous to the known fact in \cite[Lemma~2.4]{BDFIK}: if $\mathcal{A}$ is abelian, so is $\mathbf{F}(\mathcal{A}; \omega)$.

\begin{lem}\label{lem:F-ex}
    The class $\tilde{\mathcal{E}}$ is an exact structure on $\mathbf{F}(\mathcal{A}; \omega)$. Consequently, we have an exact category $(\mathbf{F}(\mathcal{A}; \omega), \tilde{\mathcal{E}})$.
\end{lem}

\begin{proof}
    The verification of the axioms above is routine. To show (Ex2), we note that a pullback diagram is essentially the same as a certain kernel. Therefore, by Lemma~\ref{lem:kerinF} the required pullback in $\mathbf{F}(\mathcal{A}; \omega)$ exists and is given componentwise. Dually, one verifies (Ex2)$^{\rm  op}$.
\end{proof}

\begin{lem}\label{lem:proj}
Assume that $\mathcal{A}$ has enough projective objects. Then so does $\mathbf{F}(\mathcal{A}; \omega)$. Moreover, an $\omega$-factorization is projective if and only if it is a direct summand of $\theta^0(P)\oplus \theta^1(Q)$ for some projective objects $P$ and $Q$ in $\mathcal{A}$.
\end{lem}

\begin{proof}
    By the adjunctions (\ref{iso:theta0}) and (\ref{iso:theta1}), both $\theta^0(P)$ and $\theta^1(Q)$ are projective for any projective objects $P, Q$ in $\mathcal{A}$. Let $X=(X^0, X^1; d_X^0, d_X^1)$ be any $\omega$-factorization. Take two deflations $a\colon P\rightarrow X^0$ and $b\colon Q\rightarrow X^1$ with $P, Q$ projective in $\mathcal{A}$. Then we obtain two morphisms in $\mathbf{F}(\mathcal{A}; \omega)$:
    $$(a, d_X^0\circ a)\colon \theta^0(P) \longrightarrow X \mbox{ and }  (\varepsilon_{X^0}\circ T^{-1}(d_X^1\circ b),\; b)\colon \theta^1(Q)\longrightarrow X.$$
    Combining these two morphisms, we obtain a deflation $\theta^0(P)\oplus \theta^1(Q)\rightarrow X$. This proves that  $\mathbf{F}(\mathcal{A}; \omega)$ has enough projective objects. Moreover, if $X$ is projective, this deflation splits. This yields the final statement.
\end{proof}

In what follows, we will assume that the autoequivalence $T$ on $\mathcal{A}$ is an exact autoequivalence.

\begin{prop}\label{prop:Frobenius}
    Let $(\mathcal{A}, \mathcal{E})$ be a Frobenius exact category. Then so is $(\mathbf{F}(\mathcal{A}; \omega), \tilde{\mathcal{E}})$.
\end{prop}

\begin{proof}
    By Lemma~\ref{lem:proj}, the exact category $\mathbf{F}(\mathcal{A}; \omega)$ has enough projective objects, and projective objects are precisely direct summands of  $\theta^0(P)\oplus \theta^1(Q)$ for some projective-injective objects $P$ and $Q$ in $\mathcal{A}$. Dually, $\mathbf{F}(\mathcal{A}; \omega)$ has enough injective objects, and injective objects are precisely direct summands of  $\theta^0(P)\oplus \theta^1(Q)$ for some projective-injective objects $P$ and $Q$ in $\mathcal{A}$; here, we need to use the adjunctions (\ref{iso:theta01}) and (\ref{iso:theta11}). When applying (\ref{iso:theta11}), we use the assumption that $T$ is exact. Then we conclude that  the exact category $\mathbf{F}(\mathcal{A}; \omega)$ is Frobenius.
\end{proof}

Denote by $\mathbf{F}^0(\mathcal{A}; \omega)$ the full subcategory of $\mathbf{F}(\mathcal{A}; \omega)$ formed by $\omega$-factorizations $X=(X^0, X^1; d_X^0, d_X^1)$ whose first components $X^1$ are projective in $\mathcal{A}$. By the proof above, it is also a Frobenius exact category with the same projecitve-injective objects as $\mathbf{F}(\mathcal{A}; \omega)$.  Denote by $ \underline{\mathbf{F}}(\mathcal{A}; \omega)$ and  $\underline{\mathbf{F}}^0(\mathcal{A}; \omega)$ their stable categories.

The following result is analogous to \cite[Proposition~6.3]{Chen24}.

\begin{prop}
Let $\mathcal{A}$ be a Frobenius category.   Then there is a recollement.
    \[\xymatrix{
 \underline{\mathbf{F}}^0(\mathcal{A}; \omega)\;\ar[rr]|-{\rm inc}&&\;  \underline{\mathbf{F}}(\mathcal{A}; \omega) \; \ar[rr]|-{{\rm pr}^1}
  \ar@/^1.5pc/[ll]|{{\rm inc}_\rho}\ar@/_1.5pc/[ll]|{{\rm inc}_\lambda}&&
  \;\underline{\mathcal{A}} \ar@/^1.5pc/[ll]|{\theta^0}\ar@/_1.5pc/[ll]|{\theta^1}
}\]
    Here, ``${\rm inc}$" denotes the inclusion functor.
\end{prop}

\begin{proof}
    By Lemma~\ref{lem:3adj}(2), we have the adjoint pairs $(\theta^1, {\rm pr}^1)$ and $({\rm pr}^1, \theta^0)$ between the unstable categories. Moreover, both $\theta^s$ are fully faithful. The essential kernel of ${\rm pr}^1\colon  \underline{\mathbf{F}}(\mathcal{A}; \omega)\rightarrow \underline{\mathcal{A}}$ is $ \underline{\mathbf{F}}^0(\mathcal{A}; \omega)$.  Then the required recollement follows immediately from the one in Lemma~\ref{lem:rec}.
\end{proof}

\begin{rem}
    Let $X=(X^0, X^1; d_X^0, d_X^1)$ be an $\omega$-factorization. By Lemma~\ref{lem:rec}, ${\rm inc}_\lambda (X)$ is isomorphic to the cone of $((\varepsilon_{X^0})^{-1}\circ T^{-1}(d_X^1), \; {\rm Id}_{X^1})\colon \theta^1(X^1)\rightarrow X$. To compute the cone, we take a deflation $x\colon P\rightarrow X^0$ with $P$ projective, which induces a morphism $(x, d_X^1\circ x)\colon \theta^0(P)\rightarrow X$. Then $\Sigma^{-1}({\rm inc}_\lambda (X))$ is isomorphic to the kernel of the following induced deflation
    $$\theta^1(X^1)\oplus \theta^0(P)\longrightarrow X.$$
    Since the kernel is computed componentwise, its zeroth component is the kernel of $T^{-1}(X^1)\oplus P\rightarrow X^0$ and its first component is $P$. In particular, it belongs to $\mathbf{F}^0(\mathcal{A}; \omega)$. Dually, $\Sigma ({\rm inc}_\rho(X))$ is isomorphic to the cokernel of the following inflation
    $$X\longrightarrow \theta^0(X^1)\oplus S\theta^0(Q),$$
    where we take an inflation $X^0\rightarrow Q$ with $Q$ projective.
\end{rem}

The following notion is essentially due to \cite[Definition~4.1]{Chen24}.

\begin{defn}\label{defn:p-null1}
A morphism $f=(f^0, f^1)\colon X\rightarrow Y$ between two $\omega$-factorizations is called \emph{p-null-homotopical} provided that there exist two morphisms $h^0\colon X^0 \rightarrow T^{-1}(Y^1)$ and $h^1\colon X^1\rightarrow Y^0$, which factor through projective objects in $\mathcal{A}$ and satisfy the following identities:
$$f^0=h^1\circ d_X^0+ \varepsilon_{Y^0}\circ T^{-1}(d_Y^1)\circ h^0 \mbox{ and } f^1=d_Y^0\circ h^1+(\eta_{Y^1})^{-1}\circ T(h^0)\circ d_X^1.$$
Here, we use the adjunction quadruple (\ref{quad:1}).
\end{defn}

The following is analogous to \cite[Lemma~4.2]{Chen24}.

\begin{lem}\label{lem:p-null}
Let  $f\colon X\rightarrow Y$ be a morphism between two $\omega$-factorizations. Then $f$ is p-null-homotopical if and only if it factors through an object of the form $\theta^0(P)\oplus \theta^1(Q)$ for projective objects $P$ and $Q$ in $\mathcal{A}$.
\end{lem}

Consequently, the morphisms in $ \underline{\mathbf{F}}(\mathcal{A}; \omega)$ are equivalent classes of morphisms in $\mathbf{F}(\mathcal{A}; \omega)$ modulo p-null-homotopical morphisms.

\begin{proof}
For the ``only if" part, we assume that $f$ is p-null-homotopical. We assume that $h^0$ and $h^1$ factor as follows:
$$h^0\colon X^0\xrightarrow{a} T^{-1}(Q)\xrightarrow{T^{-1}(b)} T^{-1}(Y^1) \mbox{ and } h^1\colon X^1\xrightarrow{x} P \xrightarrow{y} Y^0.$$
Then we have two composite morphisms of $\omega$-factorizations:
$$X\xrightarrow{(a, \; (\eta_Q)^{-1}\circ T(a)\circ d_X^1)} \theta^1(Q) \xrightarrow{(\varepsilon_{Y^0}\circ T^{-1}(d_Y^1\circ b), \; b)} Y \mbox{ and } X \xrightarrow{(x\circ d_X^0, x)} \theta^0(P) \xrightarrow{(y, d_Y^0\circ y)} Y.$$
Their sum yields $f$, which proves that $f$ factors through $\theta^0(P)\oplus \theta^1(Q)$. Here, we implicitly uses the adjunctions in Lemma~\ref{lem:3adj}. For the ``if" part, we just reverse the argument.
\end{proof}

\section{Module factorizaions and idempotents}\label{sec:modf}

In this section, we study factorizations in module categories, with an emphasis on factorizations with Gorenstein projective components. We will see that GP-compatible idempotents play a role.

Let $A$ be an arbitrary  ring.  Fix an element $\omega\in A$ and an automorphism $\sigma\colon A\rightarrow A$ such that  $\omega a=\sigma(a)\omega$ for each $a\in A$ and that $\sigma(\omega)=\omega$. In particular, the element $\omega$ is \emph{normal} in $A$, that is, $A\omega=\omega A$.

We define the \emph{twisted  matrix ring} $M_2(A; \omega, \sigma)$ as follows: as an abelian group, it is same as $M_2(A)$; its multiplication is given by
$$\begin{pmatrix}
    a_{11} & a_{12}\\
    a_{21} & a_{22}
\end{pmatrix} \begin{pmatrix}
    b_{11} & b_{12}\\
    b_{21} & b_{22}
\end{pmatrix}=\begin{pmatrix}
    a_{11}b_{11}+a_{12}b_{21}\omega & a_{11} b_{12}+a_{12}b_{22}\\
    a_{21}\sigma(b_{11})+a_{22}b_{21} & a_{21}\omega
b_{12}+a_{22}b_{22}\end{pmatrix}.$$

We are most interested in the following situation.

\begin{rem}
    Assume that $\omega\in A$ is \emph{regular} and normal. The regularity condition means that for nonzero $a\in A$, both the products $a \omega$ and $\omega a$ are nonzero.  It follows that   there is a unique automorphism $\sigma$ on $A$ satisfying $\omega a=\sigma(a)\omega$ for each $a\in A$. Moreover, we have $\sigma(\omega)=\omega$. In this situation, the twisted matrix ring $M_2(A; \omega, \sigma)$ is isomorphic to the  subring $\begin{pmatrix}
       A & A\\
       A\omega & A
    \end{pmatrix}$ of the usual matrix ring $M_2(A)$; see \cite[Section~3]{Chen24}.
\end{rem}

For each $A$-module $M$, the twisted $A$-module ${^\sigma(M)}$ is defined as follows. As an abelian group, we have ${^\sigma(M)}=M$, whose typical element is written as ${^\sigma(m)}$. The left $A$-action is defined such that
$$a \; {^\sigma(m)}={^\sigma(\sigma(a)m)}$$
for any $a\in A$. This gives rise to the \emph{twisting endofunctor} ${^\sigma(-)}$ on $A\mbox{-Mod}$, which is  an autoequivalence. Its quasi-inverse is given by the twisting endofunctor ${^{(\sigma^{-1})}(-)}$ with respect to $\sigma^{-1}$.

The element $\omega$ induces a natural transformation
$$\omega\colon {\rm Id}_{A\mbox{-}{\rm Mod}} \longrightarrow {^\sigma(-)}$$
defined such that $\omega_M(m)={^\sigma(\omega m)}$ for any $A$-module $M$ and $m\in M$. Here, we abuse the notation.
We consider the category $\mathbf{F}(A\mbox{-Mod}; \omega)$ of $\omega$-factorizations. These $\omega$-factorizations are called \emph{module factorizations} in \cite{Chen24}; in the commutative case, they are called linear factorizations in \cite{DM}; see also \cite{BT}. We mention the work \cite{SZ} on $n$-fold  module factorizations.

To simplify the notation, we set
$$\mathbf{F}(A; \omega)=\mathbf{F}(A\mbox{-Mod}; \omega), \; \mathbf{GF}(A; \omega)=\mathbf{F}(A\mbox{-GProj}; \omega) \mbox{ and } \mathbf{MF}(A; \omega)=\mathbf{F}(A\mbox{-proj}; \omega).$$
Furthermore, we set
$$\mathbf{GF}^0(A; \omega)=\mathbf{F}^0(A\mbox{-GProj}; \omega).$$
 The objects in $\mathbf{MF}(A; \omega)$ are called   \emph{matrix factorizations}  in  \cite{CCKM}; see also \cite{MU}. The objects in $\mathbf{GF}^0(A; \omega)$ are analogous to the generalized matrix factorizations in \cite[Section~9]{EP}. By Proposition~\ref{prop:Frobenius}, all of $\mathbf{GF}(A; \omega)$, $\mathbf{GF}^0(A; \omega)$ and $\mathbf{MF}(A; \omega)$ are Frobenius exact.

 Consider the cartesian product category $A\mbox{-Mod}\times A\mbox{-Mod}$. Then we have the underlying functor
$${\rm pr}=({\rm pr}^0, {\rm pr}^1)\colon \mathbf{F}(A; \omega)\longrightarrow A\mbox{-Mod}\times A\mbox{-Mod},$$
which sends an $\omega$-factorization $X=(X^0, X^1; d_X^0, d_X^1)$ to the underlying pair $(X^0, X^1)$ of $A$-modules.

The following results slightly strengthen the ones in \cite[Section~3]{Chen24}.

\begin{prop}\label{prop:M-2}
    Keep the assumptions above. Then there is an equivalence of categories
    $$\mathbf{F}(A; \omega)\simeq M_2(A; \omega, \sigma)\mbox{-}{\rm Mod},$$
    which restricts to two equivalences of exact categories:
   $$\mathbf{GF}(A; \omega)\simeq M_2(A; \omega, \sigma)\mbox{-}{\rm GProj} \mbox{ and } \mathbf{GF}^0(A; \omega)\simeq M_2(A; \omega, \sigma)\mbox{-}{\rm GProj}^{e_{11}}.$$
   Moreover, the idempotent $e_{11}=\begin{pmatrix}
       1 & 0\\
       0 & 0
   \end{pmatrix}$ of $M_2(A; \omega, \sigma)$ is GP-compatible.
\end{prop}

\begin{proof}
    Write $\Gamma=M_2(A; \omega, \sigma)$. We identify $A$ with $e_{11}\Gamma e_{11}$ and $(1_\Gamma -e_{11})\Gamma (1_\Gamma -e_{11})$. Let $X=(X^0, X^1; d_X^0, d_X^1)$ be any $\omega$-factorization. In particular, both $X^i$ are $A$-modules. We define a $\Gamma$-module $\Phi(X)=\begin{pmatrix}
        X^1\\ X^0
    \end{pmatrix}$ as follows: its typical element is written as a column vector $\begin{pmatrix}
        x^1\\ x^0
    \end{pmatrix}$; the left $\Gamma$-action is given by
    $$\begin{pmatrix}
        a_{11} & a_{12}\\
        a_{21} & a_{22}
    \end{pmatrix} \begin{pmatrix}
        x^1\\ x^0
    \end{pmatrix}=\begin{pmatrix}
        a_{11}x^1+a_{12} d_X^0(x^0)\\
        a_{21} \delta_{X^0}(d_X^1(x^1))+a_{22}x^0
    \end{pmatrix}. $$
    Here, $\delta_{X^0}\colon {^\sigma{(X^0)}}\rightarrow X^0$ sends $^\sigma{(x^0)}$  to $x^0$, which  is an isomorphism of abelian groups.  This gives rise to a functor $\Phi\colon \mathbf{F}(A; \omega)\rightarrow \Gamma\mbox{-Mod}$.

    Conversely, let $N$ be a $\Gamma$-module. We have a $\omega$-factorization
    $$\Psi(N)=((1_\Gamma -e_{11})N, e_{11}N; d^0, d^1),$$
    where $d^0\colon (1_\Gamma -e_{11})N\rightarrow e_{11}N$ sends $x$ to $\begin{pmatrix}
        0& 1 \\0 & 0
    \end{pmatrix}x$, and $d^1\colon e_{11}N\rightarrow {^\sigma((1_\Gamma -e_{11})N)}$ sends $y$ to $^\sigma(\begin{pmatrix}
        0 & 0\\ 1 & 0
    \end{pmatrix}y)$. This yields another functor $\Psi\colon \Gamma\mbox{-Mod}\rightarrow  \mathbf{F}(A; \omega)$. It is routine to verify that these two functors are quasi-inverse to each other.

    By Lemma~\ref{lem:3adj} the underlying functor ${\rm pr}$ is a Frobenius functor \cite{Mor65}. We apply \cite[Theorem~3.2]{Chen-Ren} to the composite functor ${\rm pr}\circ \Psi$, which is a faithful Frobenius functor. It follows that a $\Gamma$-module $N$ is Gorenstein projective if and only if both the $A$-modules $(1_\Gamma-e_{11})N$ and $e_{11}N$ are Gorenstein projective. Then we have the two restricted equivalences. Moreover, since $X^1=e_{11}\Phi(X)$ for any $\omega$-factorization $X$, we infer that the idempotent $e_{11}$ is GP-compatible.
\end{proof}

Consider the quotient ring $\bar{A}=A/(\omega)$. Since $\omega$ is normal, we have $(\omega)=A\omega =\omega A$. For each $\omega$-factorization $X$,  its zeroth cokernel ${\rm Cok}^0(X)$ is defined to be the cokernel of $d_X^0\colon X^0\rightarrow X^1$. Since $\omega$ vanishes on ${\rm Cok}^0(X)$, it becomes an $\bar{A}$-module. This gives rise to the (zeroth) \emph{cokernel functor}
$${\rm Cok}^0\colon \mathbf{F}(A; \omega)\longrightarrow \bar{A}\mbox{-Mod}.$$

The following result is due to \cite[Theorem~5.6]{Chen24}, which motivates the study of the category $\mathbf{GF}^0(A; \omega)$.

\begin{prop}\label{prop:Chen24}
    Assume that the element $\omega$ is regular and normal. Then the cokernel functor induces a triangle equivalence
$${\rm Cok}^0\colon \underline{\mathbf{GF}}^0(A; \omega)\stackrel{\sim}{\longrightarrow} \bar{A}\mbox{-}\underline{\rm GProj}.$$
\end{prop}

\begin{rem}\label{rem:dense}
    By \cite[the proof of Theorem~5.6]{Chen24}, the restricted cokernel functor
    $$ {\rm Cok}^0\colon \mathbf{GF}^0(A; \omega) \longrightarrow \bar{A}\mbox{-}{\rm GProj}$$ is dense. Consequently, an $\bar{A}$-module $M$ is  Gorensstein projective if and only if there is a module factorization $X\in \mathbf{GF}^0(A; \omega)$ with $M\simeq {\rm Cok}^0(X)$.
\end{rem}

In what follows, we will show that the cokernel functor is compatible with a certain idempotent.

Let $e=e^2$ be an idempotent in $A$ satisfying $\sigma(e)=e$. Then $\sigma$ restricts to an automorphism on $eAe$. Consider the element $e\omega e\in eAe$. Then we have
$$e\omega e=e\omega =\omega e.$$
For each $a\in eAe$, we have
$$(e\omega e )a=e\omega a e=e\sigma(a)\omega e=\sigma(a) (e\omega e). $$
In particular, $e\omega e$ is normal in $eAe$. The element $e \omega e$ induces a natural transformation
$$e\omega e \colon {\rm Id}_{eAe\mbox{-}{\rm Mod}} \longrightarrow {^\sigma(-)}.$$
Here, we abuse $\sigma$ with its restriction $\sigma |_{eAe}$. Then the \emph{Schur functor}
\begin{align}\label{fun:Schur}
    S_e\colon \mathbf{F}(A; \omega) \longrightarrow \mathbf{F}(eAe; e\omega e)
\end{align}
is naturally defined, which sends an $\omega$-factorization $X=(X^0, X^1; d_X^0, d_X^1)$ to the $e\omega e$-factorization $S_e(X)=(eX^0, eX^1; d_X^0, d_X^1)$.

\begin{lem}\label{lem:eAe}
Let $\omega$ be regular and normal in $A$. Suppose that $\omega e=e\omega$. Then $e\omega e$ is regular and normal in $eAe$.
\end{lem}

\begin{proof}
    Denote by $\sigma$ the unique automorphism determined by $\omega$. Then we have $\sigma(e)=e$. By the discussion above, we infer that $e\omega e$ is normal in $eAe$. Take any element $a\in eAe$. Assume that $(e\omega e) a=0$. We have $(e\omega e) a=(\omega e)e=\omega a$. Sine $\omega$ is regular in $A$, we have $a=0$. Dually, the condition $a(e\omega e)=0$ implies $a=0$. This proves that $e\omega e$ is regular in $eAe$.
\end{proof}

The following result might be viewed as a refinement of Proposition~\ref{prop:Chen24}.

\begin{prop}\label{prop:compa}
     Let $\omega$ be regular and normal in $A$, and let $e$ be an GP-compatible idempotent satisfying $\omega e=e\omega$. Then the corresponding idempotent $\bar{e}$ in $\bar{A}$ is also GP-compatible and the cokernel functor restricts to a triangle equivalence
     $$ {\rm Cok}^0  \colon \underline{\mathbf{GF}}^0(A; \omega)^e \stackrel{\sim}{\longrightarrow} \bar{A}\mbox{-}\underline{\rm GProj}^{\bar{e}}.$$
\end{prop}

Here, $\mathbf{GF}^0(A; \omega)^e$ denotes the full subcategory of $\mathbf{GF}^0(A; \omega)$ formed by thoses $\omega$-factorizations $X$ such that $S_e(X)$ are projective-injective  in $\mathbf{GF}^0(eAe; e\omega e)$. In particular, the $eAe$-module $eX^0$ has to be projective.

\begin{proof}
    We observe that $A\omega \cap eAe=eAe(e\omega e)$. It follows that the surjection $eAe\rightarrow \bar{e}\bar{A}\bar{e}$ induces an isomorphism of rings.
    $$eAe/{(e\omega e)}\simeq \bar{e}\bar{A}\bar{e}$$
    We identify these two rings. Similarly, by $eA\cap A\omega=(e\omega e)eA$, we obtain an isomorphism of $\bar{e}\bar{A}\bar{e}$-modules
    $$\bar{e}\bar{A}\simeq \bar{e}\bar{A}\bar{e}\otimes_{eAe} eA.$$
    Since $eA$ is a projective $eAe$-modules, the $\bar{e}\bar{A}\bar{e}$-module $\bar{e}\bar{A}$ is also projective.

    Recall from Lemma~\ref{lem:eAe} that $e\omega e$ is regular and normal in $eAe$. We observe a commutative diagram of functors.
   \begin{align}\label{diag:S}
       \xymatrix{
  \mathbf{GF}^0(A; \omega)\ar[d]_-{S_e}  \ar[rr]^-{{\rm Cok}^0} && \bar{A}\mbox{-GProj} \ar[d]^-{S_{\bar{e}}}\\
\mathbf{GF}^0(eAe; e\omega e)  \ar[rr]^-{{\rm Cok}^0} && \bar{e}\bar{A}\bar{e}\mbox{-GProj}}
   \end{align}
    Here, we use the fact that these two horizontal  functors are dense; see Remark~\ref{rem:dense}. It follows that the restricted Schur functor $S_{\bar{e}}$ associated to the idempotent $\bar{e}$ is well defined. In other words, for any Gorenstein projective $\bar{A}$-module $G$, the $\bar{e}\bar{A}\bar{e}$-module $S_{\bar{e}}(G)=\bar{e}G$ is Gorenstein projective. We infer that $\bar{e}$ is GP-compatible.

    Consider the corresponding diagram of (\ref{diag:S}) among the stable categories.
      \begin{align*}
       \xymatrix{
  \underline{\mathbf{GF}}^0(A; \omega)\ar[d]_-{S_e}  \ar[rr]^-{{\rm Cok}^0} && \bar{A}\mbox{-\underline{GProj}} \ar[d]^-{S_{\bar{e}}}\\
\underline{\mathbf{GF}}^0(eAe; e\omega e)  \ar[rr]^-{{\rm Cok}^0} && \bar{e}\bar{A}\bar{e}\mbox{-\underline{GProj}}}
   \end{align*}
   By Proposition~\ref{prop:Chen24} the two  horizontal functors are equivalences. It follows that the upper cokernel functor induces an equivalence between the essential kernels of the vertical functors. Then the required equivalence follows immediately.
\end{proof}

\section{Two-dimensional factorizations}\label{sec:2-dim}

We introduce two-dimensional factorizations with respect to two commuting natural transformations.  When the  category of underlying objects is Frobenius, so is the category of such factorizations; see Proposistion~\ref{prop:Frobenius-2}. However, the corresponding p-null-homotopical morphisms are much more complicated; see Definition~\ref{defn:p-null-2}.

\subsection{The general setting}\label{subsec:6.1} Let $\mathcal{C}$ be a category.  We fix an autoequivalence $T$ and a natural transformation $\omega  \colon {\rm Id}_\mathcal{C}\rightarrow T$ satisfying $\omega T=T\omega$. Let $T'\colon \mathcal{C}\rightarrow \mathcal{C}$ be another autoequivalence, and $\omega'\colon {\rm Id}_\mathcal{C}\rightarrow T'$ be a natural transformation satisfying $\omega'T'=T'\omega'$.

Similar to (\ref{quad:1}), we fix the adjunction quadruple
$$(T'^{-1}, T'; \eta', \varepsilon').$$
We define a natural transformation $\omega'^{(-)}\colon T'^{-1}\rightarrow {\rm Id}_\mathcal{C}$ by $\omega'^{(-)}=\varepsilon'\circ T'^{-1}\omega'$; see Remark~\ref{rem:nat-trans}.

We assume that $\omega'$ \emph{commutes} with $\omega$ in the following sense: there is a natural isomorphism
$$\xi \colon T'T \stackrel{\sim}\longrightarrow TT'$$
such that
$$\omega T'=\xi  \circ (T'\omega) \mbox{ and } T\omega'=\xi  \circ (\omega'T).$$

\begin{rem}\label{rem:nat-iso}
    (1) The isomorphism $\xi$ induces two natural isomorphisms
    $$\xi^{(+)}\colon T^{-1}T'\longrightarrow T'T^{-1} \mbox{ and } \xi^{(-)}\colon T'^{-1}T\longrightarrow TT'^{-1},$$
    which are given as follows:
    $$T^{-1}T'\xrightarrow{T^{-1}T'\eta} T^{-1}T'TT^{-1} \xrightarrow{T^{-1}\xi T^{-1}} T^{-1}TT'T^{-1} \xrightarrow{\varepsilon T'T^{-1}} T'T^{-1},$$
    and
    $$T'^{-1}T\xrightarrow{T'^{-1}T\eta'} T'^{-1}TT' T'^{-1}\xrightarrow{T'^{-1}\xi^{-1}T'^{-1}} T'^{-1}T'TT'^{-1}\xrightarrow{\varepsilon'TT'^{-1}} T T'^{-1}. $$

    (2) Later, we will need the following four natural isomorphisms:
\begin{align*}
u^{0, 0}=\varepsilon'\circ \varepsilon T'^{-1}T' \circ T^{-1} \xi^{(-)} T'& \colon T^{-1}T'^{-1}TT'\longrightarrow {\rm Id}_\mathcal{C};\\
u^{1, 0}=\varepsilon'\circ \eta^{-1}T'^{-1}T' &\colon TT^{-1}T'^{-1}T'\longrightarrow {\rm Id}_\mathcal{C};\\
u^{0, 1}=  \eta'^{-1} \circ T'\varepsilon T'^{-1} \circ T'T^{-1} \xi^{(-)} &\colon   T'T^{-1}T'^{-1}T\longrightarrow {\rm Id}_\mathcal{C};\\
u^{1, 1}=  \eta'^{-1}\circ T'\eta^{-1}T'^{-1}\circ \xi^{-1} T^{-1}T'^{-1} &\colon TT'T^{-1}T'^{-1}\longrightarrow {\rm  Id}_\mathcal{C}.
\end{align*}
    We mention that $u^{1,1}=\eta'^{-1}\circ \eta^{-1}T'T'^{-1}\circ T (\xi^{(+)})^{-1}T'^{-1}$.
\end{rem}

 We consider the category $\mathbf{F}(\mathcal{C}; \omega)$ of $\omega$-factorizations.
The pair $(T', \xi )$ gives rise to an autoequivalence
$$\tilde{T'}\colon \mathbf{F}(\mathcal{C}; \omega)\longrightarrow \mathbf{F}(\mathcal{C}; \omega),$$
which sends an $\omega$-factorization $X=(X^0, X^1; d_X^0, d_X^1)$ to
$$\tilde{T'}(X)=(T'(X_0), T'(X_1); T'(d_X^0), \xi_{X^0}\circ T'(d_X^1)),$$
and a morphism $f=(f^0, f^1)$ to $\tilde{T'}(f)=(T'(f_0), T'(f_1))$. To verify that $\tilde{T'}(X)$ is an $\omega$-factorization, we use the condition $\omega T'=\xi  \circ (T'\omega)$. The natural transformation $\omega'\colon {\rm Id}_\mathcal{C}\rightarrow T'$ induces  further a new natural transformation
$$\tilde{\omega}' \colon {\rm Id}_{\mathbf{F}(\mathcal{C}; \omega)}\longrightarrow \tilde{T'},$$
which is defined such that
$$\tilde{\omega}'_X=(\omega'_{X^0}, \omega'_{X^1})\colon X\longrightarrow \tilde{T'}(X).$$
To verify that this is indeed a morphism between $\omega$-factorizations, one uses $T\omega'=\xi \circ (\omega'T)$. Finally, we observe that
$$\tilde{\omega}' \tilde{T'}=\tilde{T'}\tilde{\omega}'.$$

We consider the category $\mathbf{F}(\mathbf{F}(\mathcal{C}; \omega);\tilde{\omega}')$ of $\tilde{\omega}'$-factorizations in $\mathbf{F}(\mathcal{C}; \omega)$. Its objects are given by the following commutative diagrams,  whose rows are $\omega$-factorizations and columns are $\omega'$-factorizations.
\begin{align}\label{diag:fact}
    \xymatrix{
T'(X^{0,0}) \ar[rr]^-{T'(d_h^{0,0})} && T'(X^{1,0}) \ar[rr]^-{\xi_{X^{0,0}}\circ T'(d_h^{1,0})} && TT'(X^{0,0})\\
X^{0,1}\ar[u]^-{d_v^{0,1}} \ar[rr]^-{d_h^{0,1}} && X^{1,1} \ar[u]^-{d_v^{1,1}}\ar[rr]^-{d_h^{1,1}} && T(X^{0,1}) \ar[u]_-{T(d_v^{0,1})}\\
X^{0,0}\ar[u]^-{d_v^{0,0}} \ar[rr]^-{d_h^{0,0}} && X^{1,0} \ar[u]^-{d_v^{1,0}}\ar[rr]^-{d_h^{1,0}} && T(X^{0,0}) \ar[u]_-{T(d_v^{0,0})}
}
\end{align}
Here, to be more precise, the rightmost column is isomorphic to an $\omega'$-factorization via the isomorphism $\xi_{X^{0,0}}$. Such a diagram will be called an \emph{$(\omega, \omega')$-factorization along $\xi$}, and is denoted by $X^{\bullet, \bullet}$. Alternatively, we have an $\tilde{\omega}'$-factorization
$$X^{\bullet, \bullet}=(X^{\bullet, 0}, X^{\bullet, 1}; d_v^{\bullet, 0}, d_v^{\bullet, 1}).$$
We view such an  $(\omega, \omega')$-factorization as a \emph{two-dimensional factorization}, which is visualized as a commutative  square with curved arrows.
\begin{align}\label{diag:2-dim}
\xymatrix{
 X^{0, 1}\ar@<+.7ex>[rr]^-{d_h^{0, 1}}    \ar@<+.7ex>@{~>}[dd]^-{d_v^{0,1}} && X^{1, 1} \ar@<+.7ex>@{~>}[ll]^-{d_h^{1,1}}  \ar@<+.7ex>@{~>}[dd]^-{d_v^{1,1}}\\ \\
 X^{0, 0}\ar@<+.7ex>[rr]^-{d_h^{0, 0}} \ar@<+.7ex>[uu]^-{d_v^{0, 0}}  && X^{1, 0}  \ar@<+.7ex>[uu]^-{d_v^{1, 0}}\ar@<+.7ex>@{~>}[ll]^-{d_h^{1,0}}
}
\end{align}
Here, the subscript $h$ means `horizontal' and $v$ means `vertical'. The commutativity means that the four commutative squares in (\ref{diag:fact}).

We write
$$\mathbf{F}(\mathcal{C};\omega, \omega')=\mathbf{F}(\mathbf{F}(\mathcal{C}; \omega);\tilde{\omega}'),$$
which is referred as the category of $(\omega, \omega')$-factorizations along $\xi$.

Alternatively, we may form the category $\mathbf{F}(\mathcal{C};\omega')$ of $\omega'$-factorizations first. The pair $(T, \xi^{-1})$ gives rise to an autoequivalence
$$\tilde{T}\colon \mathbf{F}(\mathcal{C};\omega')\longrightarrow \mathbf{F}(\mathcal{C};\omega'),$$
which sends an $\omega'$-factorization $Y=(Y^0, Y^1; d_Y^0, d_Y^1)$ to
$$\tilde{T}(Y)=(T(Y^0), T(Y^1); T(d_Y^0), \xi^{-1}_{Y^0}\circ T(d_Y^1)).$$
Moreover, we have a natural transformation
$$\tilde{\omega}\colon {\rm Id}_{\mathbf{F}(\mathcal{C};\omega')} \longrightarrow \tilde{T}$$
induced by $\omega$. We form the category of \emph{$(\omega', \omega)$-factorizations along $\xi^{-1}$}
$$\mathbf{F}(\mathcal{C}; \omega', \omega)=\mathbf{F}(\mathbf{F}(\mathcal{C}; \omega');\tilde{\omega}).$$

We have the following symmetry property.

\begin{prop}\label{prop:symmetry}
    Keep the assumptions above. Then we have an isomorphism of categories.
    $$\mathbf{F}(\mathcal{C}; \omega, \omega')\simeq \mathbf{F}(\mathcal{C}; \omega', \omega)$$
\end{prop}

\begin{proof}
    We observe that objects in both $\mathbf{F}(\mathcal{C}; \omega, \omega')$ and $\mathbf{F}(\mathcal{C}; \omega, \omega')$ are represented by commutative diagrams of the shape (\ref{diag:fact}) or commutative  squares with curved arrows. Then the required isomorphism is obtained by the transpose of diagrams or squares.
\end{proof}

 In what follows, we will concentrate  on $\mathbf{F}(\mathcal{C}; \omega, \omega')$. Similar to (\ref{theta:1}), we have two functors
$$\tilde{\theta}^s\colon \mathbf{F}(\mathcal{C}, \omega)\longrightarrow  \mathbf{F}(\mathcal{C}; \omega, \omega')$$
for $s=0, 1$, which assign to each $\omega$-factorization $X$ two \emph{trivial $\tilde{\omega}'$-factorizations} $\tilde{\theta}^s(X)$. Consequently, for $0\leq s,t\leq 1$, we have four functors
$$\theta^{s,t}=\tilde{\theta}^t \circ \theta^s\colon \mathcal{C}\longrightarrow \mathbf{F}(\mathcal{C}; \omega, \omega').$$
For each object $C$ in $\mathcal{C}$, the $(s,t)$-entry of $\theta^{s,t}(C)$ is precisely $C$. For example, $\theta^{1,1}(C)$ is depicted as the following commutative diagram.
\[\xymatrix{
T'T'^{-1}T^{-1}(C) \ar[rr]^-{T'T'^{-1}(\omega^{(-)}_C)} && T' T'^{-1}(C) \ar[rr]^-{} && TT'T'^{-1}T^{-1}(C)\\
T^{-1}(C) \ar[u]^-{\eta'_{T^{-1}(C)}} \ar[rr]^-{\omega^{(-1)}_C} && C \ar[u]^-{\eta'_{C}}\ar[rr]^-{\eta_C} && TT^{-1}(C)\ar[u]_-{T(\eta'_{T^{-1}(C)})} \\
T'^{-1}T^{-1}(C) \ar[u]^-{\omega'^{(-)}_{T^{-1}(C)}} \ar[rr]^-{T'^{-1}(\omega^{(-)}_C)} && T'^{-1}(C) \ar[u]^-{\omega'^{(-)}_C} \ar[rr]^-{\xi^{(-)}_{T^{-1}(C)}\circ T'^{-1}(\eta_C)} && TT'^{-1}T^{-1}(C)  \ar[u]_-{T(\omega'^{(-)}_{T^{-1}(C)})}
}\]
Here, for $\xi^{(-)}$ we refer to Remark~\ref{rem:nat-iso}(1). The unnamed arrow at the top is given by $\xi_{T'^{-1}T^{-1}(C)} \circ T'(\xi^{(-)}_{T^{-1}(C)}\circ T'^{-1}(\eta_C))$.

\subsection{Exact structures}  In what follows, we replace the category  $\mathcal{C}$ above by an exact category $(\mathcal{A}, \mathcal{E})$ and assume that $T$ is an exact autoequivalence on $\mathcal{A}$. We declare that a composable pair
$$X^{\bullet, \bullet}\xrightarrow{f^{\bullet, \bullet}} Y^{\bullet, \bullet} \xrightarrow{g^{\bullet, \bullet}} Z^{\bullet, \bullet}$$
in $\mathbf{F}(\mathcal{A}; \omega, \omega')$ is a conflation if each component belongs to $\mathcal{E}$. By Lemma~\ref{lem:F-ex}, this gives rise to an exact structure on $\mathbf{F}(\mathcal{A}; \omega, \omega')$.

\begin{prop}\label{prop:Frobenius-2}
    Assume that $\mathcal{A}$ has enough projective objects. Then so does $\mathbf{F}(\mathcal{A}; \omega, \omega')$. Moreover, an object $X^{\bullet, \bullet}$ is projective if and only if it is isomorphic to a direct summand of $\bigoplus_{0\leq i, j\leq 1} \theta^{i,j}(P)$ for some projective object $P$. Furthermore, if $\mathcal{A}$ is Frobenius, so is $\mathbf{F}(\mathcal{A}; \omega, \omega')$.
\end{prop}

\begin{proof}
    This is obtained by applying Lemma~\ref{lem:proj} and Proposition~\ref{prop:Frobenius}, repeatedly.
\end{proof}

We denote by  $\underline{\mathbf{F}}(\mathcal{A}; \omega, \omega')$ the stable category. To understand its morphisms, we will consider the corresponding p-null-homotopical morphisms in $\mathbf{F}(\mathcal{A}; \omega, \omega')$.

\begin{defn}\label{defn:p-null-2}
    Let $f^{\bullet, \bullet}\colon X^{\bullet, \bullet}\rightarrow Y^{\bullet, \bullet}$ be a morphism in $\mathbf{F}(\mathcal{A}; \omega, \omega')$. We say that $f^{\bullet, \bullet}$ is \emph{p-null-homotopical} provided that there exist four morphisms $s^{0,0}\colon X^{0,0}\rightarrow T^{-1}T'^{-1}(Y^{1, 1})$, $s^{1, 0}\colon X^{1, 0}\rightarrow T'^{-1}(Y^{0, 1})$, $s^{0, 1}\colon X^{0, 1}\rightarrow T^{-1}(Y^{1, 0})$ and $s^{1, 1}\colon X^{1,1}\rightarrow Y^{0, 0}$ such that the following conditions are satisfied.
    \begin{enumerate}
    \item[(1)] The four morphisms $s^{i,j}$ factor through projective objects in $\mathcal{A}$.
    \item[(2)] We have
    \begin{align*}
    f^{0,0}=&s^{1,1}\circ d_v^{1,0}\circ d_h^{0,0}+\varepsilon'_{Y^{0,0}}\circ T'^{-1}(\partial_v^{0,1})\circ s^{1,0}\circ d_h^{0, 0}\\
    &+\varepsilon_{Y^{0, 0}} \circ T^{-1}(\partial_h^{1,0})\circ s^{0,1}\circ d_v^{0,0}+u^{0, 0}_{Y^{0, 0}}\circ T^{-1}T'^{-1}(T(d_v^{0,1})\circ d_h^{1,1})\circ s^{0,0},\\
f^{1,0}=& \eta^{-1}_{Y^{1,0}}\circ T(s^{0,1})\circ T(d_v^{0,0})\circ d_h^{1,0}+u^{1, 0}_{Y^{1, 0}}\circ TT^{-1}T'^{-1}(\partial_v^{1,1}) \circ T(s^{0, 0})\circ d_h^{1, 0}\\
    &+\partial_h^{0,0}\circ s^{1,1}\circ d_v^{1,0}+\varepsilon'_{Y^{1,0}}\circ T'^{-1}(\partial_v^{1,1}\circ \partial_h^{0,1})\circ s^{1,0},\\
    f^{0,1}=&(\eta')^{-1}_{Y^{1,0}}\circ T'(s^{1,0})\circ d_v^{1,1}\circ d_h^{0,1}+\partial_v^{0, 0} \circ s^{1,1}\circ d_h^{0,1}\\
    &+u^{0, 1}_{Y^{0, 1}}\circ T'T^{-1}T'^{-1}(\partial_h^{1,1})\circ T'(s^{0,0})\circ d_v^{0,1}+\varepsilon_{Y^{0,1}}\circ T^{-1}(T(\partial_v^{0, 0})\circ \partial_h^{1,0})\circ s^{0,1},\\
    \mbox{and} & \\
    f^{1,1}=&  u^{1,1}_{Y^{1, 1}}\circ TT'(s^{0,0})\circ T(d_v^{0,1})\circ d_h^{1,1}+ \eta^{-1}_{Y^{1,1}}\circ \partial_v^{1,0}\circ T(s^{0,1})\circ d_h^{1,1},\\
    &+(\eta')^{-1}_{Y^{1,1}}\circ T'T'^{-1}(\partial_h^{0,1})\circ T'(s^{1,0})\circ d_v^{1,1}+\partial_v^{1,0}\circ \partial_h^{0,0}\circ s^{1,1}.
    \end{align*}
    \end{enumerate}
    Here, $Y^{\bullet, \bullet}=(Y^{\bullet, 0}, Y^{\bullet, 1}; \partial_v^{\bullet, 0}, \partial_v^{\bullet, 1})$ and the natural isomorphisms $u^{s, t}$ are defined in Remark~\ref{rem:nat-iso}.
\end{defn}

\begin{lem}\label{lem:p-null-2}
     Let $f^{\bullet, \bullet}\colon X^{\bullet, \bullet}\rightarrow Y^{\bullet, \bullet}$ be a morphism in $\mathbf{F}(\mathcal{A}; \omega, \omega')$. Then $f^{\bullet, \bullet}$ is p-null-homotopical if and only if it factors through $\bigoplus_{0\leq i, j\leq 1} \theta^{i,j}(P)$ for some projective object $P$ in $\mathcal{A}$.
\end{lem}

\begin{proof}
We have the induced adjunction quadruple $((\tilde{T}')^{-1}, \tilde{T}'; \tilde{\eta}', \tilde{\varepsilon}')$ on $\mathbf{F}(\mathcal{A}, \omega)$ from  the one $(T'^{-1}, T'; \eta', \varepsilon')$ on $\mathcal{A}$.

We observe that the following fact holds: the morphism $f^{\bullet, \bullet}$ is p-null-homotopical if and only if there are morphisms $h^{\bullet, 0}\colon X^{\bullet, 0}\rightarrow (\tilde{T}')^{-1} (Y^{\bullet, 1})$ and $h^{\bullet, 1} \colon X^{\bullet,1} \rightarrow Y^{\bullet, 0}$ in $\mathbf{F}(\mathcal{A}, \omega)$ such that
$$f^{\bullet, 0}= h^{\bullet, 1}\circ d_v^{\bullet, 0}+ \tilde{\varepsilon}'_{Y^{\bullet, 0}} \circ (\tilde{T}')^{-1}(\partial_v^{\bullet, 1})\circ h^{\bullet, 0}$$
and
$$f^{\bullet, 1}=\partial_v^{\bullet, 0} \circ h^{\bullet, 1}+ (\tilde{\eta}'_{Y^{\bullet, 1}})^{-1} \circ \tilde{T}'(h^{\bullet, 0})\circ d_v^{\bullet, 1};$$
moreover, both $h^{\bullet, 0}$ and $h^{\bullet, 1}$ are p-null-homotopical in $\mathbf{F}(\mathcal{A}, \omega)$. In view of Definition~\ref{defn:p-null1}, we infer that the morphisms $s^{0, 0}$ and $s^{1, 0}$ correspond to $h^{\bullet, 0}$, and that the morphisms $s^{0, 1}$ and $s^{1,1}$ correspond to $h^{\bullet, 1}$. We omit the details. Now, the required result follows by applying Lemma~\ref{lem:p-null} twice.
\end{proof}

In what follows, we assume that the exact category $\mathcal{A}$ is Frobenius. We consider the full subcategory
$$\mathbf{F}^0(\mathcal{A}; \omega, \omega')=\mathbf{F}^0(\mathbf{F}(\mathcal{A}; \omega); \tilde{\omega}')$$
of $\mathbf{F}(\mathcal{A}; \omega, \omega')$. We mention that $X^{\bullet, \bullet}$ belongs to $\mathbf{F}^0(\mathcal{A}; \omega, \omega')$ if and only if its upper row $X^{\bullet, 1}$ is a projective object in $\mathbf{F}(\mathcal{A}; \omega)$. Consequently, the object $X^{1, 1}$ is projective in $\mathcal{A}$. Hence, the following \emph{transpose-projection functor}
\begin{align}\label{fun:tpr}
    {\rm tpr}^1\colon \mathbf{F}^0(\mathcal{A}; \omega, \omega')\longrightarrow \mathbf{F}^0(\mathcal{A};\omega'), \; X^{\bullet, \bullet} \mapsto (X^{1, 0}, X^{1,1}; d_v^{1, 0}, d_v^{1,1})
\end{align}
is well defined.

We define the full subcategory
$$\mathbf{F}^{0,0}(\mathcal{A}; \omega, \omega')$$ of $\mathbf{F}^0(\mathcal{A}; \omega, \omega')$ formed by those $X^{\bullet,\bullet}$ satisfying that $(X^{1, 0}, X^{1,1}; d_v^{1, 0}, d_v^{1,1})$ is projective in $\mathbf{F}^0(\mathcal{A};\omega')$. Therefore, in the corresponding square (\ref{diag:2-dim}) of $X^{\bullet, \bullet}$, any ($1$-dimensional) factoriztions around $X^{1, 1}$ are projective. In particular, the objects $X^{0, 1}$, $X^{1, 0}$ and $X^{1, 1}$ are all projective in $\mathcal{A}$.

Both exact categories $\mathbf{F}^0(\mathcal{A}; \omega, \omega')$ and $\mathbf{F}^{0, 0}(\mathcal{A}; \omega, \omega')$ are Frobenius. The stable category
$$\underline{\mathbf{F}}^{0,0}(\mathcal{A}; \omega, \omega')$$
coincides with the essential kernel of the induced transpose-projection functor
$${\rm tpr}^1\colon \underline{\mathbf{F}}^0(\mathcal{A}; \omega, \omega')\longrightarrow \underline{\mathbf{F}}^0(\mathcal{A};\omega').$$

\begin{rem}
(1) Similarly, we set $\mathbf{F}^0(\mathcal{A}; \omega', \omega)=\mathbf{F}^0(\mathbf{F}(\mathcal{A}; \omega'); \tilde{\omega})$.  We mention the following \emph{asymmetry}: the isomorphism in Proposition~\ref{prop:symmetry} does not restrict to an isomorphism between  $\mathbf{F}^0(\mathcal{A}; \omega, \omega')$ and $\mathbf{F}^0(\mathcal{A}; \omega', \omega)$.

(2) By exchanging $\omega$ and $\omega'$, we may also define the subcategory $\mathbf{F}^{0,0}(\mathcal{A}; \omega', \omega)$ of $\mathbf{F}^0(\mathcal{A}; \omega', \omega)$. We emphasize that the isomorphism in Proposition~\ref{prop:symmetry}  does restrict to an isomorphism, that is, a \emph{restricted symmetry}
$$\mathbf{F}^{0,0}(\mathcal{A}; \omega, \omega')\simeq \mathbf{F}^{0,0}(\mathcal{A}; \omega', \omega).$$
\end{rem}

\section{Regular sequences of length two} \label{sec:7}

In this section, we  prove Theorem~C in the case $n=2$.

Let $A$ be an arbitrary  ring, and $\omega', \omega \in A$. Recall that  a \emph{regular sequence} $(\omega', \omega)$ of length two in $A$ consists of  a regular normal element $\omega'$ in $A$ , and another element  $\omega \in A$ satisfying that the corresponding element $\bar{\omega}=\omega+(\omega')$ is regular normal in $\bar{A}=A/{(\omega')}$; see \cite{KKZ}.

Let $\sigma'$ and $\sigma$ be two automorphisms on $A$, and $\xi$ an invertible element in $A$. The triple $(\sigma', \sigma; \xi)$ is called a \emph{type} if
\begin{align}\label{equ:type}
\sigma'\sigma(a)=\xi \; (\sigma\sigma'(a)) \; \xi^{-1}, \; \mbox{ for all } a\in A.
\end{align}

\begin{defn}\label{defn:type}
Let $(\sigma', \sigma; \xi)$ be  a type.    A regular sequence $(\omega', \omega)$ is called  \emph{of type $(\sigma', \sigma; \xi)$} provided that the following conditions are satisfied:
\begin{enumerate}
    \item[(T1)] $\omega' a=\sigma'(a)\omega'$ for all $a\in A$, and $\sigma'(\omega')=\omega'$;
    \item[(T2)] $\omega a=\sigma(a)\omega$ for all $a\in A$, and $\sigma(\omega)=\omega$;
    \item[(T3)] $\sigma'(\omega)=\xi\omega$ and $\sigma(\omega')=\xi^{-1}\omega'$.
\end{enumerate}
\end{defn}

\begin{rem}
    Since $\omega'$ is regular normal, by (T1) the automorphism $\sigma'$ on $A$ is uniquely determined by $\omega'$. By (T3), the automorphism $\sigma$ on $A$ induces an automorphism $\bar{\sigma}$ on $\bar{A}$. Since $\bar{\omega}$ is regular normal in $\bar{A}$, by (T2) the induced automorphism $\bar{\sigma}$ is uniquely determined by $\bar{\omega}$. In general, we do not know the uniqueness of the automorphism $\sigma$.
\end{rem}

In what follows, we fix a regular sequence $(\omega', \omega)$ of type $(\sigma', \sigma; \xi)$. The elements $\omega$ and $\omega'$ induce natural transformations
$$\omega\colon {\rm Id}_{A\mbox{-}{\rm  Mod}}\longrightarrow {^\sigma(-)} \mbox{ and } \omega'\colon {\rm Id}_{A\mbox{-}{\rm  Mod}} \longrightarrow {^{\sigma'}(-)},$$
which are given such that $\omega_M(x)={^\sigma(\omega x)}$ and $\omega'_M(x)={^{\sigma'}(\omega'x)}$ for any $A$-module $M$  and $x\in M$. The invertible element $\xi$ induces a natural isomorphism
$$\xi\colon {^{\sigma'}(-)}{^{\sigma}(-)}= {^{\sigma\sigma'}(-)}\longrightarrow {^{\sigma'\sigma}(-)}={^{\sigma}(-)} {^{\sigma'}(-)},$$
which is given such that  $\xi_M$ sends ${^{\sigma\sigma'}(x)}$ to ${^{\sigma'\sigma}(\xi x)}$. In other words, $\omega'$ commutes with $\omega$ via $\xi$. The assumptions in Section~\ref{sec:2-dim} hold. We form the category of \emph{two-dimensional module factorizations}
$$\mathbf{F}(A; \omega, \omega')=\mathbf{F}(A\mbox{-Mod}; \omega, \omega').$$
Moreover, we obtain the Frobenius categories
$$\mathbf{GF}(A;\omega, \omega')=\mathbf{F}(A\mbox{-GProj}; \omega, \omega') \mbox{ and } \mathbf{GF}^{0,0}(A; \omega, \omega')=\mathbf{F}^{0, 0}(A\mbox{-GProj}; \omega, \omega'),$$
which consist of two-dimensional  module factorizations with Gorenstein projective components.

\begin{lem}\label{lem:Cok-2}
    Let $X^{\bullet, \bullet}=(X^{\bullet, 0}, X^{\bullet, 1}; d_v^{\bullet, 0}, d_v^{\bullet, 1})$ be an object in $\mathbf{GF}(A; \omega, \omega')$. Then the following sequence of $A$-modules
$$0\longrightarrow X^{0, 0}\xrightarrow{\begin{pmatrix} - d_h^{0,0} \\
                                                       d_v^{0, 0}
                                                       \end{pmatrix}}
        X^{1,0}\oplus X^{0, 1} \xrightarrow{\begin{pmatrix}d_v^{1, 0}& d_h^{0, 1} \end{pmatrix}} X^{1, 1}$$
    is exact; moreover, the cokernel of $\begin{pmatrix}d_v^{1, 0}& d_h^{0, 1} \end{pmatrix}$ belongs to $A/{(\omega, \omega')} \mbox{-{\rm GProj}}$.
\end{lem}

We define the \emph{total-cokernel} ${\rm TCok}(X^{\bullet, \bullet})$ of $X^{\bullet, \bullet}$ to be the cokernel of $\begin{pmatrix}d_v^{1, 0}& d_h^{0, 1} \end{pmatrix}$.

\begin{proof}
Due to the following commutative square,  the sequence above is a complex.
\[\xymatrix{
X^{0, 1} \ar[rr]^-{d_h^{0, 1}} && X^{1,1}\\
X^{0,0} \ar[u]^-{d_v^{0, 0}} \ar[rr]^-{d_h^{0,0}} && X^{1, 0} \ar[u]_-{d_v^{1, 0}}
}\]
The element $\omega'$ is regular in $A$. Consequently,  $\omega'$ is a non-zerodivisor on any Gorenstein projective $A$-module. Both $d_v^{0, 0}$ and $d_v^{1, 0}$ fit into module factorizations in $\mathbf{GF}(A; \omega')$. It follows that they are monomorphisms. Consequently, the complex above is quasi-isomorphic to the following one
$$\bar{X}^{0, 1} \xrightarrow{\bar{d}_h^{0, 1}}   \bar{X}^{1, 1},$$
which is induced from $d_{h}^{0, 1}$. Here, $\bar{X}^{0, 1}$ is the cokernel of $d_v^{0, 0}$,  and $\bar{X}^{1, 1}$ is the cokernel of $d_v^{1,0}$. By \cite[Theorem~2.13]{Chen24}, both $\bar{X}^{0, 1}$  and $\bar{X}^{1, 1}$ belongs to $\bar{A}\mbox{-GProj}$ with $\bar{A}=A/{(\omega')}$. The morphism $\bar{d}_h^{0, 1}$ fits into a module factorization in $\mathbf{GF}(\bar{A};  \bar{\omega})$. Since $\bar{\omega}$ is regular in $\bar{A}$, we infer that $\bar{d}_h^{0, 1}$ is a monomorphism; moreover, by  \cite[Theorem~2.13]{Chen24} again, its cokernel belongs to $A/{(\omega, \omega')} \mbox{-{\rm GProj}}$. Then the required results follow immediately.
\end{proof}

By Lemma~\ref{lem:Cok-2}, we infer that the following  total-cokernel functor
$${\rm TCok}\colon \mathbf{GF}(A; \omega, \omega') \longrightarrow  A/{(\omega, \omega')} \mbox{-{\rm GProj}}, \; X^{\bullet, \bullet} \mapsto {\rm Cok}\begin{pmatrix}d_v^{1, 0}& d_h^{0, 1} \end{pmatrix}$$
is well defined. Moreover, it is exact and sends projective objects to projective modules. Therefore, it induces a triangle functor
$${\rm TCok}\colon \underline{\mathbf{GF}}(A; \omega, \omega')\longrightarrow  A/{(\omega, \omega')} \mbox{-\underline{\rm GProj}}.$$

\begin{thm}\label{thm:2-dim}
     Let  $(\omega', \omega)$ be a regular sequence of type $(\sigma', \sigma; \xi)$ in $A$. Set $B=A/{(\omega, \omega')}$. Then  the functor ${\rm TCok}$ above restricts to a triangle equivalence
     $${\rm TCok}\colon \underline{\mathbf{GF}}^{0, 0}(A; \omega, \omega') \stackrel{\sim}\longrightarrow  B\mbox{-\underline{\rm GProj}}.$$
     When $A$ is left noetherian, the equivalence above restricts further to a triangle equivalence
     $${\rm TCok}\colon \underline{\mathbf{MF}}^{0, 0}(A; \omega, \omega') \stackrel{\sim}\longrightarrow  B\mbox{-\underline{\rm Gproj}}^{<+\infty}.$$
\end{thm}

Here, $\underline{\mathbf{MF}}^{0, 0}(A; \omega, \omega')=\mathbf{F}^{0, 0}(A\mbox{-proj}; \omega, \omega')$ and $B\mbox{-Gproj}^{<+\infty}$ denotes the category of finitely generated Gorenstein projective $B$-modules which have finite projective dimension as $A$-modules.

We make some preparation for the proof. Recall the twisted matrix ring $\Gamma=M_2(A; \omega, \sigma)$. The automorphism $\sigma'$ on $A$ induces an automorphism
$$\tilde{\sigma}'\colon \Gamma \longrightarrow \Gamma, \; \begin{pmatrix} a_{11} & a_{12} \\
                                                                    a_{21} & a_{22}\end{pmatrix} \mapsto \begin{pmatrix} \sigma'(a_{11}) & \sigma'(a_{12}) \\
                                                                    \sigma'(a_{21})\xi & \sigma'(a_{22})\end{pmatrix}. $$
For the verification, we use (\ref{equ:type}) and Definition~\ref{defn:type}(T3). The diagonal matrix $\omega' I_2={\rm diag}(\omega', \omega')$ is a regular normal element in $\Gamma$. Indeed, we have
$$(\omega' I_2)x=\tilde{\sigma}'(x) (\omega' I_2)$$
for each element $x\in \Gamma$. Consequently, the element $\omega'I_2$ in $\Gamma$ induces a natural transformation
\begin{align}\label{nat:I-2}
    \omega' I_2\colon {\rm Id}_{\Gamma \mbox{-} {\rm Mod} } \longrightarrow {^{\tilde{\sigma}'}(-)}.
\end{align}

Recall from Subsection~\ref{subsec:6.1} that the natural transformation
$$\omega'\colon {\rm Id}_{A \mbox{-} {\rm Mod}} \longrightarrow {^{\sigma'}(-)}=T'$$
extends to a new natural transformation
\begin{align}\label{nat:omega'}
\tilde{\omega}'\colon  {\rm Id}_{\mathbf{F}(A; \omega)} \longrightarrow \tilde{T}'.
\end{align}
 By the proof of Proposition~\ref{prop:M-2}, we have an explicit equivalence
$$\Phi\colon \mathbf{F}(A; \omega) \stackrel{\sim}\longrightarrow \Gamma\mbox{-Mod}.$$

The following lemma shows that the two natural transformations (\ref{nat:I-2}) and (\ref{nat:omega'}) are compatible.
\begin{lem}\label{lem:kappa}
    There is a natural isomorphism $\kappa \colon \Phi \tilde{T}'\rightarrow {^{\tilde{\sigma}'}(-)}\Phi$ such that
    $$\kappa \circ \Phi \tilde{\omega}'=(\omega' I_2)\Phi.$$
\end{lem}

\begin{proof}
For each $\omega$-factorization $X=(X^0, X^1; d_X^0, d_X^1)$, $\kappa_X$ identifies $\begin{pmatrix}
   {^{\sigma'}(X^1)} \\ {^{\sigma'}(X^0)}
\end{pmatrix}$ with ${^{\tilde{\sigma}'}{\begin{pmatrix}
   X^1 \\ X^0
\end{pmatrix}}}$ in a canonical manner. The remaining verification is trivial.
\end{proof}

Recall that $\bar{A}=A/{(\omega')}$.  We observe a canonical isomorphism of rings.
\begin{align}\label{iso:Gamma}
    \Gamma/{(\omega'I_2)} \stackrel{\sim} \longrightarrow M_2(\bar{A}; \bar{\omega}, \bar{\sigma})
\end{align}
Moreover, the idempotent $e_{11}=\begin{pmatrix}
    1 & 0\\
    0 & 0
\end{pmatrix}$ in $\Gamma$ corresponds to the one $\bar{e}_{11}=\begin{pmatrix}
    \bar{1} & 0\\
    0 & 0
\end{pmatrix}$ in $M_2(\bar{A}; \bar{\omega}, \bar{\sigma})$.  In view of Proposition~\ref{prop:M-2}, the idempotent $\bar{e}_{11}$ is GP-compatible.

\vskip 5pt

\noindent \emph{Proof of Theorem~\ref{thm:2-dim}.}\quad The pair $(\Phi, \kappa)$ in Lemma~\ref{lem:kappa} induces the triangle equivalence $\Phi_*$ in the upper row of the following square.
\[\xymatrix{
\underline{\mathbf{F}}^0(\mathbf{GF}(A; \omega); \tilde{\omega}') \ar[d]_-{{\rm tpr}^1} \ar[rr]^-{\Phi_*} && \underline{\mathbf{GF}}^0(\Gamma; \omega' I_2) \ar[rr]^-{\rm Cok} && M_2(\bar{A}; \bar{\omega}, \bar{\sigma})\mbox{-\underline{GProj}} \ar[d]^-{S_{\bar{e}_{11}}} \\
\underline{\mathbf{GF}}^0(A; \omega') \ar[rrrr]^{\rm Cok} &&&& \bar{A}\mbox{-\underline{GProj}}
}\]
Here, ${\rm tpr}^1$ denotes the transpose-projection functor (\ref{fun:tpr}) and $S_{\bar{e}_{11}}$ is the Schur functor  associated to $\bar{e}_{11}$.    For the upper cokernel functor, we use Proposition~\ref{prop:Chen24} and the ring isomorphism (\ref{iso:Gamma}). The square above commutes up to a natural isomorphism.

Taking the essential kernels of the vertical functors, we obtain a restricted triangle equivalence induced by ${\rm Cok}\circ \Phi_*$.
\begin{align}\label{equiv:1}
    {\rm Cok}\circ \Phi_*\colon \underline{\mathbf{GF}}^{0, 0}(A; \omega, \omega') \stackrel{\sim}\longrightarrow M_2(\bar{A}; \bar{\omega}, \bar{\sigma})\mbox{-}\underline{\rm GProj}^{\bar{e}_{11}}
\end{align}

Applying Proposition~\ref{prop:M-2}, we have another triangle equivalence
\begin{align}\label{equiv:2}
   M_2(\bar{A}; \bar{\omega}, \bar{\sigma})\mbox{-}\underline{\rm GProj}^{\bar{e}_{11}}\simeq \underline{\mathbf{GF}}^0(\bar{A}; \bar{\omega}).
\end{align}
Finally, by Proposition~\ref{prop:Chen24} we have a triangle equivalence
\begin{align}\label{equiv:3}
    {\rm Cok}\colon \underline{\mathbf{GF}}^0(\bar{A}; \bar{\omega})\stackrel{\sim}\longrightarrow B\mbox{-}\underline{\rm GProj}.
\end{align}
Here, we identify $B$ with $\bar{A}/{(\bar{\omega})}$.

We combine the three equivalences  (\ref{equiv:1})-(\ref{equiv:3}) above and obtain a composite  triangle equivalence
$$\underline{\mathbf{GF}}^{0, 0}(A; \omega, \omega')\stackrel{\sim}\longrightarrow B\mbox{-}\underline{\rm GProj}.$$
We use ${\rm Cok}$  in both (\ref{equiv:1}) and (\ref{equiv:3}). The two cokernel functors  give rise to the required  total-cokernel functor ${\rm TCok}$. In other words, the composite functor above coincides with ${\rm TCok}$. This completes the proof of the first equivalence.

Assume that $A$ is left noetherian. The same argument as above yields an equivalence
$${\rm TCok}\colon \underline{\mathbf{F}}^{0, 0}(A\mbox{-Gproj}; \omega, \omega') \stackrel{\sim}\longrightarrow  B\mbox{-\underline{\rm Gproj}}.$$
Take any object $X^{\bullet, \bullet} \in \mathbf{F}^{0, 0}(A\mbox{-Gproj}; \omega, \omega')$. From the very definition, we observe that the components $X^{0, 1}$, $X^{1, 0}$ and $X^{1,1}$ belong to $A\mbox{-proj}$. By Lemma~\ref{lem:Cok-2}, ${\rm TCok}(X^{\bullet, \bullet})$ has finite projective dimension as an $A$-module if and only if so does $X^{0, 0}$. Since the $A$-module $X^{0, 0}$ is Gorenstein projective, it is projective if and only if it has finite projective dimension. Consequently, ${\rm TCok}(X^{\bullet, \bullet})$ belongs to $B\mbox{-Gproj}^{<+\infty}$ if and only if $X^{0, 0}$ belongs to $A\mbox{-proj}$, or equivalently, $X^{\bullet, \bullet}$ belongs to $\mathbf{MF}^{0, 0}(A; \omega, \omega')$. This implies  the second  equivalence.
\hfill $\square$

\vskip 5pt

We apply Theorem~\ref{thm:2-dim} to a concrete example.

\begin{exm}
{\rm Let $k$ be a field and $A=k[x, y]$ be the polynomial algebra in  two variables. Let $f(t)\in k[t]$  be a nonzero polynomial in one variable. Then $(f(x), f(y))$ is a regular sequence in $A$, which is certainly of type $({\rm Id}_A, {\rm Id}_A; 1_A)$. Set $B=A/(f(x), f(y))$ to be the quotient algebra. Since $A$ has finite global dimension and $B$ is a finite dimensional selfinjective algebra, we have $B\mbox{-Gproj}^{<+\infty}=B\mbox{-mod}$. Consequently, we have a triangle equivalence
$${\rm TCok}\colon \underline{\mathbf{MF}}^{0, 0}(A; f(x), f(y)) \stackrel{\sim}\longrightarrow  B\mbox{-\underline{\rm mod}}.$$

Assume that $f(y)-f(x)=(y-x)\Delta$ with $\Delta=\Delta(x, y)$ a polynomial in $x$ and $y$.  We have a concrete two-dimensional matrix factorization  $X^{\bullet, \bullet}$ of $(f(x), f(y))$ as follows.
\[\xymatrix{
A\oplus A \ar@<+.7ex>[rrrr]^-{\begin{small}\begin{pmatrix} f(x) & -\Delta\\
 0 & 1\end{pmatrix}\end{small}}    \ar@<+.7ex>[ddd]^-{\begin{small}\begin{pmatrix} f(x) & -\Delta \\ y-x & 1\end{pmatrix}\end{small}} &&&& A\oplus A \ar@<+.7ex>[llll]^-{\begin{small}\begin{pmatrix} 1 & \Delta \\ 0 & f(x) \end{pmatrix}\end{small}}  \ar@<+.7ex>[ddd]^-{\begin{small}\begin{pmatrix}1 & 0 \\ 0 & f(y)\end{pmatrix} \end{small}}\\ \\ \\
 A\oplus A \ar@<+.7ex>[rrrr]^-{\begin{small}\begin{pmatrix} 1 & 0\\ x-y & f(x) \end{pmatrix}\end{small}} \ar@<+.7ex>[uuu]^-{\begin{small}\begin{pmatrix} 1 & \Delta\\ x-y & f(x)\end{pmatrix}\end{small}}  &&&& A\oplus A  \ar@<+.7ex>[uuu]^-{\begin{small}\begin{pmatrix}f(y) & 0\\ 0 & 1 \end{pmatrix} \end{small} }\ar@<+.7ex>[llll]^-{\begin{small}\begin{pmatrix}f(x) & 0\\ y-x & 1 \end{pmatrix}\end{small}}
}\]
The matrix factorization $X^{\bullet, \bullet}$ above belongs to $\mathbf{MF}^{0, 0}(A; f(x), f(y))$, whose total-cokernel is isomorphic to the $B$-module $M=B/(x-y)$. We trace the proof of Theorem~\ref{thm:2-dim} and use the computation in \cite[Section~4]{Chen24} to obtain $X^{\bullet, \bullet}$. Applying Lemma~\ref{lem:Cok-2} to $X^{\bullet, \bullet}$, we obtain a projective resolution of $M$ as an $A$-module, which is clearly non-minimal.

Set $\Lambda=k[t]/(f(t))$. Then $B$ is isomorphic to the enveloping algebra $\Lambda\otimes_k \Lambda^{\rm op}$ of $\Lambda$. Therefore, we identify $B$-modules with $\Lambda$-$\Lambda$-bimodules.  The $B$-module $M$ above corresponds to the regular $\Lambda$-$\Lambda$-bimodule $\Lambda$. The consideration here might be related to the study of Hochschild cohomology of $\Lambda$; see \cite{THolm}.
}\end{exm}

\section{The general case}\label{sec:8}

In this section, we extend Theorem~\ref{thm:2-dim} to the general case by induction.

Let $\mathcal{A}$ be any category. Assume that $\{T_i\;|\; 1\leq i\leq n\}$ is  a family of autoequivalences on $\mathcal{A}$. For each $i$, we fix a natural transformation $\omega_i\colon {\rm Id}_\mathcal{A} \rightarrow T_i$ satisfying  $T_i \omega_i=\omega_i T_i$. For any $1\leq i<j\leq n$, there is a natural isomorphism
$$\xi_{ij}\colon T_i T_j \stackrel{\sim}\longrightarrow T_j T_i$$
satisfying
$$\omega_j T_i= \xi_{ij}\circ T_i\omega_j \mbox{ and } T_j \omega_i=\xi_{ij}\circ \omega_i T_j.$$

By induction and using the data $\{T_i, \omega_i, \xi_{ij}\}_{1\leq i<j\leq n}$, we will define the category $\mathbf{F}(\mathcal{A}; \omega_1, \cdots, \omega_n)$.

We first form the category $\mathbf{F}(\mathcal{A}; \omega_n)$ of $\omega_n$-factorizations. For any $1\leq i\leq n-1$, by Subsection~\ref{subsec:6.1} the natural transformation $\omega_i$ induces a new one
$$\tilde{\omega}_i\colon {\rm Id}_{\mathbf{F}(\mathcal{A}; \omega_n)}\longrightarrow \tilde{T}_i.$$
Moreover, the natural isomorphism $\xi_{ij}$ also induces a new one
$$\tilde{\xi}_{i j}\colon \tilde{T}_i \tilde{T}_j \stackrel{\sim}\longrightarrow \tilde{T}_j \tilde{T}_i$$
for $1\leq i<j \leq n-1$. Now, using the new data $\{\tilde{T}_i,\tilde{\omega}_i, \tilde{\xi}_{ij}\}_{1\leq i<j\leq n-1}$, we define inductively
$$\mathbf{F}(\mathcal{A}; \omega_1, \cdots, \omega_n):=\mathbf{F}(\mathbf{F}(\mathcal{A}; \omega_n); \tilde{\omega}_1, \cdots, \tilde{\omega}_{n-1}). $$

We mention that any object in $\mathbf{F}(\mathcal{A}; \omega_1, \cdots, \omega_n)$ is viewed as an \emph{$n$-dimensional factorization}, which might be visualized a commutative $n$-cube $X$ with curved arrows. To be more precise,  the underlying object is given by a map
$$X\colon \{0, 1\}^{n}\longrightarrow \mathcal{A}, \; \alpha \mapsto X^\alpha. $$
Each edge in  direction $i$ is given by an $\omega_i$-factorization:
\[\xymatrix{
X^\alpha \xrightarrow{d_i^{\alpha}}  X^{\alpha+\epsilon_i} \xrightarrow{d_i^{\alpha+\epsilon_i}} T_i(X^\alpha),  \mbox{  or  }  X^{\alpha}\ar@<+.7ex>[rr]^-{d_i^\alpha} && X^{\alpha+\epsilon_i} \ar@<+.7ex>@{~>}[ll]^-{d_i^{\alpha+\epsilon_i}}.
}\]
Here, the $i$-th component of $\alpha$ is assumed to be zero. The $n$-dimensional factorization is written as
$$X=(X^\alpha; d_i)_{\alpha\in \{0, 1\}^n, 1\leq i\leq n}.$$

Assume now that $\mathcal{A}$ is a Frobenius exact category. By applying Proposition~\ref{prop:Frobenius} repeatedly, we infer that $\mathbf{F}(\mathcal{A}; \omega_1, \cdots, \omega_n)$ is a Frobenius exact category. Its stable category is denoted by $\underline{\mathbf{F}}(\mathcal{A}; \omega_1, \cdots, \omega_n)$. By generalizing Proposition~\ref{prop:symmetry}, we infer that the construction does not depend on the order of $\omega_i$'s.

By induction on $n$, we will define a very subtle subcategory $\mathbf{F}^{0^n}(\mathcal{A}; \omega_1, \cdots, \omega_n)$. Here, we write $0^n$ for $0, 0,\cdots , 0$ with $n$ copies of zeros.

The cases $n=1$ and $n=2$ are already done. In general, we have the obvious \emph{transpose-projection functor}
$${\rm tpr}^1\colon \mathbf{F}(\mathbf{F}(\mathcal{A}, \omega_n); \tilde{\omega}_1, \cdots, \tilde{\omega}_{n-1})=\mathbf{F}(\mathcal{A}; \omega_1, \cdots, \omega_n)\longrightarrow \mathbf{F}(\mathcal{A}; \omega_1, \cdots, \omega_{n-1}),$$
which sends $X$ to the $(n-1)$-dimensional facet $\{X^{(a_1, \cdots, a_{n-1}, 1)}\; |\; a_i=0,1\}$. It restricts to
$${\rm tpr}^1\colon \mathbf{F}^{0^{n-1}}(\mathbf{F}(\mathcal{A}, \omega_n); \tilde{\omega}_1, \cdots, \tilde{\omega}_{n-1}) \longrightarrow \mathbf{F}^{0^{n-1}}(\mathcal{A}; \omega_1, \cdots, \omega_{n-1}).$$
 The subcategory $\mathbf{F}^{0^n}(\mathcal{A}; \omega_1, \cdots, \omega_n)$ of $\mathbf{F}^{0^{n-1}}(\mathbf{F}(\mathcal{A}, \omega_n); \tilde{\omega}_1, \cdots, \tilde{\omega}_{n-1})$ is defined such that its stable category $\underline{\mathbf{F}}^{0^n}(\mathcal{A}; \omega_1, \cdots, \omega_n)$ is the essential kernel of
$${\rm tpr}^1\colon \underline{\mathbf{F}}^{0^{n-1}}(\mathbf{F}(\mathcal{A}, \omega_n); \tilde{\omega}_1, \cdots, \tilde{\omega}_{n-1})\longrightarrow \underline{\mathbf{F}}^{0^{n-1}}(\mathcal{A}; \omega_1, \cdots, \omega_{n-1}).$$

In what follows, we write ${\bf 1}=(1, 1, \cdots, 1)\in \{0, 1\}^n$.

\begin{rem}
    By induction on $n$, one shows that for any $X$ in $\mathbf{F}(\mathcal{A}; \omega_1, \cdots, \omega_n)$, it belongs to $\mathbf{F}^{0^n}(\mathcal{A}; \omega_1, \cdots, \omega_n)$ if and only if each of its $n$ facets containing $X^{\bf 1}$ is a projective object in the relevant category of $(n-1)$-dimensional factorizations. Therefore, the definition of  $\mathbf{F}^{0^n}(\mathcal{A}; \omega_1, \cdots, \omega_n)$ is \emph{symmetric}, that is, does not depend on the order of these $\omega_i$'s.
\end{rem}

Let $A$ be an arbitrary ring. Recall that a \emph{regular sequence} in $A$ means a sequence $(\omega_1, \cdots, \omega_n)$ of elements in $A$ such that $\omega_1$ is regular normal in $A$, and the sequence $(\bar{\omega}_2, \cdots, \bar{\omega}_n)$ is regular in the quotient ring $A/{(\omega_1)}$.

Let $(\sigma_1, \cdots, \sigma_n)$  be  a sequence of automorphisms on $A$. For $1\leq i< j \leq n$, we fix an invertible element $\xi_{ij}$ in $A$. The tuple $(\sigma_1, \cdots, \sigma_n; \xi_{ij})$ is called a \emph{type} of length $n$,  if
$$\sigma_i \sigma_j(a)=\xi_{ij} \;(\sigma_j \sigma_i(a))\;  (\xi_{ij})^{-1}$$
holds for any $a\in A$ and  $1\leq i< j \leq n$.

\begin{defn}\label{defn:type-n}
Let $(\sigma_1, \cdots, \sigma_n; \xi_{ij})$ be  a type of length $n$.    A regular sequence $(\omega_1, \cdots, \omega_n)$ is said to be \emph{of type $(\sigma_1, \cdots, \sigma_n; \xi_{ij})$} provided that the following conditions are satisfied:
\begin{enumerate}
    \item[(T1)] $\omega_i a=\sigma_i(a)\omega_i$ and $\sigma_i(\omega_i)=\omega_i$ for any $a\in A$ and each $i$;
    \item[(T2)] $\sigma_i(\omega_j)=\xi_{ij}\omega_j$ and $\sigma_j(\omega_i)=(\xi_{ij})^{-1}\omega_i$ for any $1\leq i<j \leq n$.
\end{enumerate}
\end{defn}

Set $T_i={^{\sigma_i}(-)}$ to be the twisting endofunctor on $A\mbox{-Mod}$. Each element $\omega_i$ in $A$ gives rise to a natural transformation
$$\omega_i\colon {\rm Id}_{A\mbox{-}{\rm Mod}} \longrightarrow T_i.$$
Moreover, the invertible elements $\xi_{ij}$ yield the natural isomorphism
$$\xi_{ij}\colon T_iT_j\stackrel{\sim}\longrightarrow T_jT_i.$$
These data allow us to form the category of \emph{$n$-dimensional module factorizations}
$$\mathbf{F}(A; \omega_1, \cdots, \omega_n)=\mathbf{F}(A\mbox{-Mod}; \omega_1, \cdots, \omega_n).$$
Furthermore, we write
$$\mathbf{GF}(A; \omega_1, \cdots, \omega_n)=\mathbf{F}(A\mbox{-GProj}; \omega_1, \cdots, \omega_n),$$ which is a Frobenius category.

For each $\alpha=(a_1, \cdots, a_n)\in \{0, 1\}^n$, we write $|\alpha|=a_1+\cdots+a_n$.  For each $1\leq i\leq n$, we set ${\rm sign}(\alpha, i)=(-1)^s$, where $s$ is the cardinality of the set $\{j\; |\; i<j\leq n, a_j=0\}$.

Set $B=A/{(\omega_1, \cdots, \omega_n)}$ to be the quotient ring.  We generalize Lemma~\ref{lem:Cok-2} as follows. For any $X\in \mathbf{GF}(A; \omega_1, \cdots, \omega_n)$ and $0\leq m\leq n-1$, we have a morphism
$$\bigoplus_{|\alpha|=m} X^\alpha \longrightarrow \bigoplus_{|\beta|=m+1} X^\beta,$$
given by ${\rm sign}(\alpha, i) d_i^\alpha \colon X^\alpha \rightarrow X^{\alpha+\epsilon_i}$ for $\alpha=(a_1, \cdots, a_n)$ satisfying $a_i=0$ and $|\alpha|=m$. Letting $m$ vary, we obtain an acyclic complex of Gorenstein projective $A$-modules.
\begin{align}\label{seq:TCok-n}
0\longrightarrow X^{\bf 0} \xrightarrow{((-1)^{n-i} d_i^{\bf 0})_{1\leq i\leq n}}  \bigoplus_{i=1}^n X^{\epsilon_i} \longrightarrow \cdots \longrightarrow \bigoplus_{i=1}^nX^{{\bf 1}-\epsilon_i}  \xrightarrow{\sum_{i=1}^n d_i^{{\bf 1}-\epsilon_i}} X^{{\bf 1}}
\end{align}
Moreover, the cokernel of the rightmost morphism, which is denoted by ${\rm TCok}(X)$ and called the \emph{total-cokernel} of $X$, is a Gorenstein projective $B$-module.

We write $\mathbf{GF}^{0^n}(A; \omega_1, \cdots, \omega_n)=\mathbf{F}^{0^n}(A\mbox{-GProj}; \omega_1, \cdots, \omega_n)$.

\begin{rem}
    We mention that if $X$ belongs to $\mathbf{GF}^{0^n}(A; \omega_1, \cdots, \omega_n)$, the sequence above is a projective presentation of ${\rm TCok}(X)$ of length $n$. In particular, $X^{\bf 0}$ is isomorphic to the $n$-th syzygy of the $A$-module ${\rm TCok}(X)$.
\end{rem}

Consequently, we have the following total-cokernel functor
$${\rm TCok}\colon \mathbf{GF}(A; \omega_1, \cdots, \omega_n)\longrightarrow B\mbox{-{\rm GProj}}, \; X \mapsto {\rm TCok}(X). $$
It induces a triangle functor
$${\rm TCok}\colon \underline{\mathbf{GF}}(A; \omega_1, \cdots, \omega_n)\longrightarrow B\mbox{-\underline{\rm GProj}} $$
between the corresponding stable categories.

\begin{thm}\label{thm:general}
    Let $(\omega_1, \cdots, \omega_n)$ be a regular sequence in $A$ of type $(\sigma_1, \cdots, \sigma_n; \xi_{i j})$. Set $B=A/{(\omega_1, \cdots, \omega_n)}$. Then the functor ${\rm TCok}$ above restricts to a triangle equivalence
    $${\rm TCok}\colon \underline{\mathbf{GF}}^{0^n}(A; \omega_1, \cdots, \omega_n)\stackrel{\sim}\longrightarrow B\mbox{-\underline{\rm GProj}}. $$
When $A$ is left noetherian, the equivalence above restricts further to a triangle equivalence
     $${\rm TCok}\colon \underline{\mathbf{MF}}^{0^n}(A; \omega_1, \cdots, \omega_n) \stackrel{\sim}\longrightarrow  B\mbox{-\underline{\rm Gproj}}^{<+\infty}.$$
\end{thm}

Here, $\underline{\mathbf{MF}}^{0^n}(A; \omega_1, \cdots, \omega_n)=\mathbf{F}^{0^n}(A\mbox{-proj}; \omega_1, \cdots, \omega_n)$.  In Theorem~C, we denote $0^n$ by the bold zero ${\bf 0}$.

\begin{proof}
Consider $\bar{A}=A/(\omega_1, \cdots, \omega_{n-1})$. Then $\bar{\omega}_n$ is a regular normal element in $\bar{A}$, and $\sigma_n$ induces an automorphism $\bar{\sigma}_n$ on $\bar{A}$. We form the twisted matrix ring $\Delta=M_2(\bar{A}; \bar{\omega}_n, \bar{\sigma}_n)$. We identify $B$ with $\bar{A}/{(\bar{\omega}_n)}$

Consider the twisted matrix ring $\Gamma=M_2(A; \omega_n ; \sigma_n)$. For each $1\leq i\leq n-1$, the automorphism $\sigma_i$ on $A$ induces an automorphism $\tilde{\sigma}_i$ on $\Gamma$, which sends
$$\begin{pmatrix} a_{11} & a_{12} \\
                                                                    a_{21} & a_{22}\end{pmatrix} \mbox{ to }  \begin{pmatrix} \sigma_i(a_{11}) & \sigma_i(a_{12}) \\
                                                                    \sigma_i(a_{21})\xi_{in} & \sigma_i(a_{22})\end{pmatrix}.$$
Furthermore,  $(\omega_1 I_2, \cdots, \omega_{n-1}I_2)$ is a regular sequence in $\Gamma$ of type $(\tilde{\sigma}_1, \cdots, \tilde{\sigma}_{n-1}; \xi_{ij} I_2)$. We have a canonical isomorphism of rings
\begin{align}\label{iso:GD}
    \Gamma/{(\omega_1 I_2, \cdots, \omega_{n-1}I_2)} \stackrel{\sim}\longrightarrow  \Delta.
\end{align}

Fix the idempotent $e=\begin{pmatrix}1 & 0 \\ 0 & 0 \end{pmatrix}$ in $\Gamma$. Its image in $\Delta$ is denoted by $\bar{e}$.  Set $\mathcal{A}=A\mbox{-GProj}$ and  $\mathcal{B}=\Gamma\mbox{-GProj}$. By Proposition~\ref{prop:M-2}, we have an equivalence of exact categories
$$\Phi\colon \mathbf{F}(\mathcal{A}; \omega_n) \stackrel{\sim} \longrightarrow \mathcal{B}.$$
It induces the equivalence $\Phi_*$ in the following commutative diagram.
\[\xymatrix{
\underline{\mathbf{F}}^{0^{n-1}}(\mathbf{F}(\mathcal{A}; \omega_n); \tilde{\omega}_1, \cdots, \tilde{\omega}_{n-1}) \ar[r]^-{\Phi_*} \ar[d]_-{{\rm tpr}^1} & \underline{\mathbf{F}}^{0^{n-1}}(\mathcal{B}; \omega_1 I_2, \cdots, \omega_{n-1}I_2) \ar[r]^-{\rm TCok} & \Delta\mbox{-\underline{GProj}} \ar[d]^-{S_{\bar{e}}}\\
\underline{\mathbf{F}}^{0^{n-1}}(\mathcal{A}; \omega_1, \cdots, \omega_{n-1}) \ar[rr]^-{\rm TCok} && \bar{A}\mbox{-\underline{GProj}}
}\]
The upper total-cokernel functor ${\rm TCok}$ uses the induction hypothesis  and the isomorphism (\ref{iso:GD}). Since we identify $\bar{e} \Delta \bar{e}$ with $\bar{A}$, we have the Schur functor $S_{\bar{e}}$. Here, we recall from Proposition~\ref{prop:M-2} that the idempotent $\bar{e}$ is GP-compatible.

Taking the essential kernels of the two vertical functors in the commutative diagram above, we obtain a triangle equivalence
$$\underline{\mathbf{F}}^{0^{n}}(\mathcal{A}; \omega_1, \cdots, \omega_{n})\stackrel{\sim}\longrightarrow \Delta\mbox{-\underline{GProj}}^{\bar{e}}.$$
In view of (\ref{equiv:2}) and (\ref{equiv:3}),  we deduce a triangle equivalence
$$\Delta\mbox{-\underline{GProj}}^{\bar{e}} \stackrel{\sim}\longrightarrow  B\mbox{-\underline{GProj}}.$$
This proves the first equivalence.

Replacing Lemma~\ref{lem:Cok-2} by the exact sequence (\ref{seq:TCok-n}), the same argument at the end of the proof of Theorem~\ref{thm:2-dim} applies here. Then we infer the second equivalence.
\end{proof}

\begin{rem}\label{rem:dense-n}
By tracing the proof above and using Remark~\ref{rem:dense}, we infer that the total-cokernel functor
    $${\rm TCok}\colon {\mathbf{GF}}^{0^n}(A; \omega_1, \cdots, \omega_n) \longrightarrow B\mbox{-{\rm GProj}} $$
    is dense. Consequently, a $B$-module $M$ is Gorenstein projective if and only if it is isomorphic to the total-cokernel of some module factorization in ${\mathbf{GF}}^{0^n}(A; \omega_1, \cdots, \omega_n)$. Similarly, a $B$-module $M$ belongs to $B\mbox{-{\rm Gproj}}^{<+\infty}$ if and only if there is an $n$-dimensional matrix factorization in ${\mathbf{MF}}^{0^n}(A; \omega_1, \cdots, \omega_n)$ whose total-cokernel is isomorphic to $M$.
\end{rem}

\section{Complete intersections}

In this section, we illustrate Theorem~\ref{thm:general} on complete intersections.

\subsection{Commutative complete intersections}\label{subsec:9.1}
Let $S$ be any ring with a regular sequence $(\omega_1, \cdots, \omega_n)$ such that each $\omega_i$ is central.  Then $(\omega_1, \cdots, \omega_n)$ is clearly a regular sequence of type $({\rm Id}_A, \cdots, {\rm Id}_A; 1)$. We consider the abelian category
$$\mathbf{F}(S; \omega_1, \cdots, \omega_n)=\mathbf{F}(S\mbox{-Mod}; \omega_1, \cdots, \omega_n)$$
of $n$-dimensional module factorizations. Each factorization is  a commutative $n$-cube $$X=(X^\alpha; d_i)_{\alpha\in \{0, 1\}^n, 1\leq i\leq n},$$
 whose components $X^\alpha$ are $S$-modules and edges in direction $i$ belong to $\mathbf{F}(S; \omega_i)$, that is,
$$X^\alpha \stackrel{d_i^\alpha} \longrightarrow X^{\alpha+\epsilon_i}\stackrel{d_i^{\alpha+\epsilon_i}}\longrightarrow  X^\alpha.$$
Here, when the $i$-th entry of $\alpha$ equals $1$, the $i$-th entry of  $\alpha+\epsilon_i$ is understood to be  $0$. The commutativity of $X$ means that all the composite homomorphisms along each shortest path from $X^\alpha$ to $X^\beta$ are the same. Equivalently, $X$ corresponds to an $n$-dimensional  directed lattice of $S$-modules, which is commutative and  $2$-periodic in each direction.

Let $Y=(Y^\alpha; \partial_i)_{\alpha\in \{0, 1\}^n, 1\leq i\leq n}$ be another such factorization. A morphism $f=(f^\alpha)_{\alpha\in \{0, 1\}^n}\colon X\rightarrow Y$ consists of homomorphisms $f^\alpha\colon X^\alpha \rightarrow Y^\alpha$ of $S$-modules, which satisfy
$$\partial_i^\alpha\circ f^\alpha =f^{\alpha+\epsilon_i}\circ d_i^\alpha.$$

The following is inspired by  Definition~\ref{defn:p-null-2}. The morphism $f$ is said to be \emph{p-null-homotopical} if there are homomorphisms $s^\alpha\colon X^\alpha \rightarrow Y^{{\bf 1}-\alpha}$ of $S$-modules for all $\alpha\in \{0, 1\}^n$ such that
\begin{enumerate}
    \item[(1)] all the homomorphisms $s^\alpha$ factor through projective $S$-modules;
    \item[(2)] for each $\alpha$, we have $f^\alpha=\sum_{\beta\in \{0, 1\}^n} \partial^{{\bf 1}-\beta, \alpha} \circ s^\beta\circ d^{\alpha, \beta}$.
\end{enumerate}
 Here, $d^{\alpha, \beta}$ is the unique homomorphism from $X^\alpha$ to $X^\beta$ obtained by composing  these $d_i$'s  along any shortest path from $X^\alpha$ to $X^\beta$ in the $n$-cube $X$. Similarly, we have the homomorphism $\partial^{{\bf 1}-\beta, \alpha}$ from $Y^{{\bf 1}-\beta}$ to $Y^\alpha$.

The following result is analogous to Lemma~\ref{lem:p-null-2}, which is proved by induction on $n$. It leads to  a better understanding of morphisms in  the stable category $\underline{\mathbf{F}}(S; \omega_1, \cdots, \omega_n)$.

\begin{lem}
Let $f\colon X\rightarrow Y$ be a morphism in $\mathbf{F}(S; \omega_1, \cdots, \omega_n)$. Then it factors through projective objects if and only if it is p-null-homotopical. \hfill $\square$
\end{lem}

In what follows, we assume that $S$ is a regular local ring. Then the quotient ring $R=S/{(\omega_1, \cdots, \omega_n)}$ is a \emph{complete intersection}; if $n=1$, it is called a \emph{hypersurface}.

Denote by $\underline{\rm MCM}(R)$ the stable category of maximal Cohen-Macaulay $R$-modules. The following result extends  Eisenbud's matrix factorization theorem \cite[Section~6]{Eis} from hypersurfaces to complete intersections.

\begin{thm}\label{thm:B}
     Let  $R=S/{(\omega_1, \cdots, \omega_n)}$ be a complete intersection. Then the total-cokernel functor induces a triangle equivalence
    $${\rm TCok}\colon  \underline{\mathbf{MF}}^{0^n}(S; \omega_1, \cdots, \omega_n) \stackrel{\sim}\longrightarrow  \underline{\rm MCM}(R). $$
\end{thm}

In Theorem~B of Introduction, we write the bold zero ${\bf 0}$ for $0^n$.

\begin{proof}
By \cite[Chapter Four, \S2]{ABr}, a finitely generated $R$-module is maximal Cohen-Macaualy if and only if it is Gorenstein projective. Since $S$ has finite global dimension, we infer that ${\rm MCM}(R)=R\mbox{-{\rm Gproj}}^{<+\infty}$.  Then we are done by Theorem~\ref{thm:general}.
\end{proof}

\begin{rem}\label{rem:PropA}
Combining the proof above and Remark~\ref{rem:dense-n}, we infer that an $R$-module is maximal Cohen-Macaulay if and only if  it is isomorphic to the total-cokernel of some $n$-dimensional matrix factorization in  ${\mathbf{MF}}^{0^n}(S; \omega_1, \cdots, \omega_n)$. In particular, this implies Proposition~A in Introduction.
\end{rem}

    Let us recall \emph{higher matrix factorizations} \cite{EP16}, HMF for short, over $S$ with respect to $(\omega_1, \cdots, \omega_n)$. Each HMF is a tuple $Z=(Z^0, Z^1; d, h_1, \cdots, h_n)$, which is given by the following data:
    \begin{enumerate}
        \item[(1)] for $i \in \{0, 1\}$, $Z^i$ is a free $S$-module of finite rank, which has a chosen filtration consisting of free $S$-modules
    $$0=Z^i_0\subseteq Z^i_1\subseteq \cdots \subseteq Z^i_{n-1}\subseteq Z^i_n=Z^i$$
    with all the factors $Z^i_{q+1}/{Z^i_q}$ free;
        \item[(2)] $d\colon Z^0\rightarrow Z^1$ is a homomorphism satisfying $d(Z^0_q)\subseteq Z^1_q$;
        \item[(3)] for each $1\leq q\leq n$, $h_q\colon Z^1_q\rightarrow Z^0_q$ is a homomorphism.
    \end{enumerate}
    These data are subject to the following conditions:
    \begin{enumerate}
        \item[(4)] the image of $\omega_q {\rm Id}_{Z^1_q}-d\circ h_q\colon Z^1_q\rightarrow Z^1_q$ is contained in $\sum_{l=1}^{q-1} \omega_l Z_q^1$;
        \item[(5)] the image of $\omega_q{\rm Id}_{Z^0_q}-h_q\circ d \colon Z_q^0\rightarrow Z^0_q$ is contained in $\sum_{l=1}^{q-1} \omega_lZ^0_q+Z_{q-1}^0$.
    \end{enumerate}
    The associated \emph{HMF module} is defined to be
    $$C(Z)={\rm Coker}(R\otimes_S d).$$

In view of (\ref{seq:TCok-n}), the following easy observation is natural.

\begin{prop}
    Each $n$-dimensional matrix factorization  $X=(X^\alpha; d_i)_{\alpha\in \{0, 1\}^n, 1\leq i\leq n}$ gives rise to an HMF
    $$Z=(\bigoplus_{i=1}^n X^{{\bf 1}-\epsilon_i}, X^{\bf 1}; \sum_{i=1}^n d_i^{{\bf 1}-\epsilon_i}, h_1, \cdots, h_n),$$
    where the filtrations are given by  $Z_q^1=X^{\bf 1}$ and $Z_q^0=\bigoplus_{i=1}^q X^{{\bf 1}-\epsilon_i}$ for $q\geq 1$, and $h_q$ is given by $Z_q^1=X^{\bf 1}\stackrel{d_q^{\bf 1}}\rightarrow X^{{\bf 1}-\epsilon_q}\subseteq Z_q^0$. Moreover, we have ${\rm TCok}(X)\simeq C(Z)$. \hfill $\square$
\end{prop}

In view of Remark~\ref{rem:PropA},  the result above implies that each maximal Cohen-Macaualy $R$-module is an HMF module; see \cite[Theorem~10.5]{EP}.

\subsection{Quantum complete intersections}

We will recall quantum complete intersections from \cite[Section~2]{AGP}.

 Let $k$ be a field. Let $\mathbf{q}=(q_{ij})\in M_n(k)$ such that $q_{ii}=1$ and $q_{ij}q_{ji}=1$. The \emph{quantum polynomial algebra} is defined to be
    $$A_{\bf q}=k\langle x_1, \cdots, x_n\rangle /{(x_ix_j-q_{ij}x_jx_i\; |\; 1\leq i, j\leq n)}.$$
    Fix a sequence ${\bf l}=(l_1, l_2, \cdots, l_n)$ of natural numbers. For each $1\leq i\leq n$, we define an algebra automorphism $\sigma_i$ on $A_{\bf q}$ such that
    $$\sigma_i(x_j)=(q_{ij})^{l_i} x_j$$
    for any $1\leq j\leq n$. Set $\xi_{ij}=(q_{ij})^{l_il_j}$ for $1\leq i< j\leq n$.

    We observe that $(x_1^{l_1}, \cdots, x_n^{l_n})$ is a regular sequence in $A_{\bf q}$ of type $(\sigma_1, \cdots, \sigma_n; \xi_{ij})$. The corresponding quotient algebra $$\Lambda_{\bf q, l}=A_{\bf q}/{(x_1^{l_1}, \cdots, x_n^{l_n})}$$
    is called the \emph{quantum complete intersection}, which  is a finite dimensional selfinjective algebra.  We mention that quantum complete intersections are studied in \cite{BE, BO}.

\begin{thm}\label{thm:qci}
The  total-cokernel functor  induces  a triangle equivalence
     $${\rm TCok}\colon \underline{\mathbf{MF}}^{0^n}(A_{\bf q}; x_1^{l_1}, \cdots, x_n^{l_n})\stackrel{\sim}\longrightarrow \Lambda_{\bf q, l}\mbox{-\underline{\rm mod}}.$$
\end{thm}

\begin{proof}
Since $\Lambda_{\bf q, l}$ is selfinjective, any module is Gorenstein projective. Since $A_{\bf q}$ has finite global dimension, we have $\Lambda_{\bf q, l} \mbox{-mod}=\Lambda_{\bf q, l} \mbox{-Gproj}^{<+\infty}$. The required equivalence follows from the second one in Theorem~\ref{thm:general}.
\end{proof}

\vskip 15pt

\noindent {\bf Acknowledgements.}  \; We thank Xiaofa Chen and Zhengfang Wang for many helpful comments.  The project is supported by National Key R$\&$D Program of China (No. 2024YFA1013801) and  National Natural Science Foundation of China (No.s 12325101 and  12131015).

\vskip 10pt

 {\footnotesize \noindent Xiao-Wu Chen\\
 School of Mathematical Sciences, University of Science and Technology of China\\
 Hefei 230026, Anhui, PR China\\
 xwchen$\symbol{64}$mail.ustc.edu.cn}

\end{document}